\pdfoutput=1
\documentclass[screen,nonacm,sigconf,acmthm=false,balance=false]{acmart}
\usepackage{amsmath}
\usepackage{amsthm}
\usepackage{xcolor}
\usepackage{colortbl,tabularx}
\usepackage{multirow}
\usepackage{bbold} %
\usepackage{thm-restate,apptools}
\makeatletter
\def\mathcolor#1#{\@mathcolor{#1}}
\def\@mathcolor#1#2#3{%
\protect\leavevmode
\begingroup
\color#1{#2}#3%
\endgroup
}
\makeatother
\newcommand{\lfrac}[2]{\ensuremath{{#1}/{#2}}}
\newcommand{\Field}{\ensuremath{\mathbb{K}}}
\DeclareMathOperator{\tensorproduct}{\ensuremath{\otimes}}
\newcommand{\Inverse}[1]{\ensuremath{{#1}^{-1}}}
\newcommand{\RR}{\ensuremath{\mathbb{R}}}

\newcommand{\BDual}[1]{\ensuremath{{{\bigl({#1}\bigr)}^{*}}}}

\newcommand{\mat}[1]{\ensuremath{\mathsf{#1}}}%
\newcommand{\matrixsize}[2]{\ensuremath{{{#1}{\times}{#2}}}}
\newcommand{\MatrixProduct}[2]{\ensuremath{{{\mat{#1}}\cdot{\mat{#2}}}}}
\newcommand{\Matrices}[2]{\ensuremath{{{#1}^{#2}}}}
\newcommand{\MatrixRank}[1]{\ensuremath{{\textup{Rank}\,{#1}}}}
\DeclareMathOperator{\Trace}{\ensuremath{\textup{Trace}}}
\newcommand{\Transpose}[1]{\ensuremath{{\mat{#1}}^{\intercal}}}

\newcommand{\vectorization}[1]{\ensuremath{\textup{vec}{#1}}}%
\newcommand{\vectorif}[2]{\ensuremath{{\left(#1\right)}_{#2}}}
\newcommand{\HadamardProduct}[2]{\ensuremath{\mat{#1}\odot\mat{#2}}}
\newcommand{\InvTranspose}[1]{{{\mat{#1}}^{-\intercal}}}
\newenvironment{smatrix}{\left[\begin{smallmatrix}}{\end{smallmatrix}\right]}
\def\triadone{brown}
\def\triadtwo{red}
\def\triadthree{blue}
\def\triadfour{green}
\def\triadfive{magenta}
\def\triadsix{gray}
\def\triadseven{violet}
\newcommand{\firstdim}{\ensuremath{m}}
\newcommand{\seconddim}{\ensuremath{k}}
\newcommand{\thirddim}{\ensuremath{n}}
\newcommand{\tensor}[1]{\ensuremath{\mathcal{#1}}}
\newcommand{\Contraction}[3]{\ensuremath{{{#1}\!\mid^{#2}_{#3}}}}

\newcommand{\LeftTensor}[1]{\ensuremath{{\mat{M}}_{#1}}}
\newcommand{\RightTensor}[1]{\ensuremath{{\mat{N}}_{#1}}}
\newcommand{\ProductTensor}[1]{\ensuremath{{\mat{O}}^{#1}}}
\newcommand{\HMRepresentation}[3]{{\ensuremath{\left[{{#1};{#2};{#3}}\right]}}}
\DeclareMathOperator{\IsotropyComposition}{\circ}
\newcommand{\Isotropy}[1]{\ensuremath{\mathsf{#1}}}
\newcommand{\IsotropyAction}[2]{\ensuremath{{#1}\diamond{#2}}}
\newcommand{\GrowthFactor}[1]{\ensuremath{{\gamma{\left(#1\right)}}}}

\newcommand{\bbigO}[1]{\ensuremath{O\bigl({#1}\bigl)}}

\newcommand{\row}[2]{{#1}_{#2}}

\newcommand{\norm}[1]{\left\|#1\right\|}

\newcommand{\textbignorm}[1]{{\bigl\|{#1}\bigr\|}}
\newcommand{\dualnorm}[1]{\left\|#1\right\|_*}
\newcommand{\xnorm}[2]{\left\|#1\right\|_{#2}}
\newcommand{\xnormexp}[3]{\left\|#1\right\|_{#2}^{#3}}
\newcommand{\maxnorm}[1]{{\left\|{#1}\right\|}_\infty}
\newcommand{\textmaxnorm}[1]{{{\|{#1}\|}_{\infty}}}
\newcommand{\textbigmaxnorm}[1]{{{\bigl\|{#1}\bigr\|}_{\infty}}}
\newcommand{\onenorm}[1]{\left\|#1\right\|_1}
\newcommand{\twonorm}[1]{\left\|#1\right\|_2}
\newcommand{\Fnorm}[1]{\left\|#1\right\|_F}
\newcommand{\threenorm}[1]{\left\|#1\right\|_3}
\newcommand{\zeronorm}[1]{\left\|#1\right\|_0}
\newcommand{\ulp}{\varepsilon}
\newcommand{\comp}[1]{{\widehat{#1}}}
\newcommand{\matr}[2]{\text{Mat}_{#2}(#1)}
\newcommand{\plinopt}{\href{https://github.com/jgdumas/plinopt}{\textsc{PLinOpt}}}

\usepackage{algorithm}
\usepackage[noend]{algpseudocode}

\algnewcommand{\IfThen}[2]{%
\State{}\algorithmicif\ #1\ \algorithmicthen\ #2}
\algnewcommand{\IfThenEnd}[2]{%
\State{}\algorithmicif\ #1\ \algorithmicthen\ #2\ \algorithmicend\ \algorithmicif}
\algnewcommand{\IfThenElse}[3]{%
\State{}\algorithmicif\ #1\ \algorithmicthen\ #2\ \algorithmicelse\ #3}
\algnewcommand{\ForDoEnd}[3][]{%
\ifthenelse{\equal{#1}{}}%
{\State{}\algorithmicfor\ #2\ \algorithmicdo\ #3\ \algorithmicend\ \algorithmicfor}
{\State\label{#1}\algorithmicfor\ #2\ \algorithmicdo\ #3\ \algorithmicend\ \algorithmicfor}
}

\usepackage{hyperref}
\makeatletter
\hypersetup{%
pdftitle={\@title},
pdfauthor={\@author},
breaklinks=true,
colorlinks=true,
linkcolor=purple,
citecolor=blue,
urlcolor=teal,
}
\makeatother
\usepackage[nameinlink,capitalize,noabbrev]{cleveref}
\Crefname{proposition}{Proposition}{Propositions}
\crefname{proposition}{Proposition}{Propositions}
\crefname{equation}{Eq.}{Eqs.}
\Crefname{equation}{Equation}{Equations}
\crefformat{footnote}{#2\footnotemark[#1]#3}
\title{Strassen's algorithm is not optimally accurate}
\author{Jean-Guillaume Dumas}
\orcid{0000-0002-2591-172X}
\affiliation{%
\institution{Universit\'e Grenoble Alpes}
\department{UMR CNRS 5224 LJK}
\streetaddress{700 avenue centrale, IMAG --- CS 40700}
\city{38058 Grenoble}
\country{France}
}
\author{Cl\'ement Pernet}
\orcid{0000-0001-6970-0417}
\affiliation{%
\institution{Univ.\ Grenoble Alpes, Grenoble INP}
\department{UMR CNRS 5224 LJK}
\streetaddress{700 avenue centrale, IMAG --- CS 40700}
\city{38058 Grenoble}
\country{France}
}
\author{Alexandre Sedoglavic}
\orcid{0000-0003-3225-8631}
\affiliation{%
\institution{Universit\'e de Lille}
\department{UMR CNRS 9189 CRISTAL}
\streetaddress{Cit\'e scientifique}
\city{59650 Villeneuve d'Ascq}
\country{France}
}
\theoremstyle{acmdefinition}
\newcounter{generalthm}

\newtheorem{theorem}[generalthm]{Theorem}

\newtheorem{definition}[generalthm]{Definition}

\newtheorem{lemma}[generalthm]{Lemma}
\newtheorem{notation}[generalthm]{Notation}

\newtheorem{corollary}[generalthm]{Corollary}
\newtheorem{remark}[generalthm]{Remark}

\begin{document}
\begin{abstract}
We propose a non-commutative algorithm for multiplying~\({\matrixsize{2}{2}}\)-matrices using~\(7\) coefficient products.
This algorithm reaches simultaneously a better accuracy in practice compared to previously known such fast~\({\matrixsize{2}{2}}\) algorithms and a time complexity bound with the best currently known leading term (obtained via alternative basis sparsification).
To build this algorithm, we consider matrix and tensor norm bounds governing the stability and accuracy of numerical matrix multiplication.
First, we reduce those bounds by minimizing a growth factor along the unique orbit of Strassen's~\({\matrixsize{2}{2}}\)-matrix multiplication tensor decomposition.
Second, we develop heuristics that minimize the number of operations required to realize a bilinear formula, while further improving its accuracy.
Third, we perform an alternative basis sparsification that improves on the time complexity constant and mostly preserves the overall accuracy.
\end{abstract}
\thanks{This material is based on work supported in part by the
  \grantsponsor{anr}{Agence Nationale de la Recherche}{https://anr.fr}
  under grants
\grantnum{anr}{\href{https://anr.fr//ProjetIA-15-IDEX-0002}{ANR-15-IDEX-0002}},
\grantnum{anr}{\href{https://anr.fr/Project-ANR-21-CE39-0006}{ANR-21-CE39-0006}},
\grantnum{anr}{\href{https://anr.fr/ProjetIA-22-PECY-0010}{ANR-22-PECY-0010}}%
}
\maketitle
\section{Introduction}\label{sec:introduction}
The first non-commutative algorithm for multiplying~\({\matrixsize{2}{2}}\)-matrices using~\(7\) coefficient products was discovered by Strassen~\cite{strassen:1969}.
It was subsequently proven that all such algorithms with~\(7\) multiplications all lie in a single isotropy orbit on Strassen's bilinear tensor decomposition~\cite{groot:1978a}.
We here study the numerical accuracy of the recursive application of these~\({\matrixsize{2}{2}}\) algorithms over the reals.
\par
We first propose a unified accuracy analysis of such recursive algorithms, generalizing some and improving on other state-of-the-art bounds~\cite{brent:1970a,bini:1980,demmel:2007a,ballard:2012a,BBDLS16,Dai:2023aa}.
Following the approach of~\cite{bini:1980}, we then seek to optimize the growth factor, a parameter governing the accuracy in these bounds, over Strassen's orbit.
Since the max-norm, producing the sharpest bounds, precludes smooth optimization, we relax the problem to optimizing a weaker growth factor in the Frobenius norm, which will later demonstrate to better reflect the practical accuracy observed in our experiments.
\par
The most efficient variants are then obtained from these bilinear formulas by minimizing the number of operations required to realize them.
Our heuristics for this make use of common sub-expression eliminations with rational coefficients, potential factorization via the kernel of the matrices defining the considered bilinear operators, as well as Tellegen's transposition principle.
\par
While preserving the complexity bound exponent of Strassen's algorithm,~\({n^{\log_{2}{\!7}}}\), those algorithms require slightly more operations, thus worsening the constant factor of the leading term.
We therefore finally propose further variants obtained by an alternative basis sparsification, similar to those introduced in~\cite{Karstadt:2017aa,Beniamini:2019aa}.
In fine, again thanks to a minimization of the number of operations required to realize them, we obtain variants having a time complexity bound with the best currently known leading term, that simultaneously improve on the accuracy (i.e.\ mostly preserving in practice the numerical accuracy with or without alternative basis sparsification).
\par
Our \textsc{c}++ tools for the minimization of the number of operations are gathered in the \plinopt~library~\cite{jgd:2024:plinopt}.
We also forked the Matlab framework of~\cite{Dai:2023aa} in~\cite{jgd:2024:mFMM} to experiment our implementations of the resulting fast and accurate~\(\matrixsize{2}{2}\) matrix multiplication algorithms.
\par
\Cref{sec:framework} presents the symmetries of matrix multiplication tensors that we will use.
In~\cref{sec:accuracyTheoriticalBound} we propose the unified error bounds on bilinear operators and matrix multiplication algorithms, highlighting how the growth factor parameter governs accuracy.
On a relaxed growth factor in norm~\(2\), we apply, in~\cref{sec:numericalStabilityMeasure}, a descent algorithm to reach some local minima and show in~\cref{sec:holder} that it lies within at most~\(2.6\%\) of the optimal.
Finally, \cref{sec:implem} presents our minimization heuristics and the obtained matrix multiplication algorithms and their associated accuracy benchmark.
\section{Matrix product seen as tensor}\label{sec:framework}
We recall here the formalism of tensor decomposition allowing to present clearly the symmetries, later used to search for more numerically accurate fast matrix multiplication algorithms in~\cref{sec:numericalStabilityMeasure}.
We start by briefly recalling tensorial representation of bilinear maps, through the example introduced by Strassen in~\cite{strassen:1969} of fast~\({\matrixsize{2}{2}}\)-matrix product, and we refer to~\cite{Landsberg:2016ab} for this framework.
The product~\({\mat{C}=\MatrixProduct{A}{B}}\) of~\({\matrixsize{2}{2}}\) matrices can be computed by Strassen algorithm using the following computations:
\begin{equation}\label{eq:StrassenMultiplicationAlgorithm}
\begin{array}{ll}
\mathcolor{\triadone}{\rho_{1}}\leftarrow{\mathcolor{\triadone}{a_{11}}(\mathcolor{\triadone}{b_{12}-b_{22}})},
&
\mathcolor{\triadfour}{\rho_{4}}\leftarrow{(\mathcolor{\triadfour}{a_{12}-a_{22}})(\mathcolor{\triadfour}{b_{21}+b_{22}})},
\\
\mathcolor{\triadtwo}{\rho_{2}}\leftarrow{(\mathcolor{\triadtwo}{a_{11}+a_{12}})\mathcolor{\triadtwo}{b_{22}}},
&
\mathcolor{\triadfive}{\rho_{5}}\leftarrow{(\mathcolor{\triadfive}{a_{11}+a_{22}})(\mathcolor{\triadfive}{b_{11}+b_{22}})},
\\
\mathcolor{\triadthree}{\rho_{3}}\leftarrow{(\mathcolor{\triadthree}{a_{21}+a_{22}}) \mathcolor{\triadthree}{b_{11}}},
&
\mathcolor{\triadseven}{\rho_{7}}\leftarrow{(\mathcolor{\triadseven}{a_{21}-a_{11}})(\mathcolor{\triadseven}{b_{11}+b_{12}})},
\\
\mathcolor{\triadsix}{\rho_{6}}\leftarrow{\mathcolor{\triadsix}{a_{22}}(\mathcolor{\triadsix}{b_{21}-b_{11}})},
&
 \begin{smatrix} c_{11} &c_{12} \\ c_{21} &c_{22} \end{smatrix}
 \!=\!
 \begin{smatrix}
 \mathcolor{\triadfive}{\rho_{5}} + \mathcolor{\triadfour}{\rho_{4}} - \mathcolor{\triadtwo}{\rho_{2}} + \mathcolor{\triadsix}{\rho_{6}} &
 \mathcolor{\triadsix}{\rho_{6}} + \mathcolor{\triadthree}{\rho_{3}} \\
 \mathcolor{\triadtwo}{\rho_{2}} + \mathcolor{\triadone}{\rho_{1}}&
 \mathcolor{\triadfive}{\rho_{5}} + \mathcolor{\triadseven}{\rho_{7}} + \mathcolor{\triadone}{\rho_{1}}- \mathcolor{\triadthree}{\rho_{3}}
 \end{smatrix}.
\end{array}\hspace{-15pt}
\end{equation}
This straight-line program (a.k.a.~\textsc{slp}) encodes the following bilinear map over a field~\(\Field\) with~\({\firstdim,\seconddim,\thirddim}\) equal to~\(2\):
\begin{equation}\label{eq:mxnTimesnxp}
\beta_{\textsc{mm}}(A,B):
\begin{array}[t]{ccl}
\Matrices{\Field}{\matrixsize{\firstdim}{\seconddim}}\times\Matrices{\Field}{\matrixsize{\seconddim}{\thirddim}}&\rightarrow&\Matrices{\Field}{\matrixsize{\firstdim}{\thirddim}},\\
(\mat{A},\mat{B})&\mapsto&\MatrixProduct{A}{B}.
\end{array}
\end{equation}
Indices~\({\firstdim,\seconddim,\thirddim}\) are kept in this section for the sake of clarity in order to distinguish easily the different spaces involved in the sequel.
\begin{definition}\label{def:FrobeniusInnerProduct}
The spaces~\({\Matrices{\Field}{\matrixsize{\cdot}{\cdot}}}\) can be endowed with the classical Frobenius inner product~\({{\langle\mat{M},\mat{N}\rangle}={\Trace({\MatrixProduct{\Transpose{M}}{\mat{N}}})}}\) that establishes an isomorphism between~\(\Matrices{\Field}{\matrixsize{\cdot}{\cdot}}\) and its dual space~\({{\bigl(\Matrices{\Field}{\matrixsize{\cdot}{\cdot}}\bigr)}^{*}}\).
\end{definition}

Frobenius inner product combines matrix product~(\ref{eq:mxnTimesnxp}) and the trilinear form~\({\Trace(\MatrixProduct{\Transpose{C}}{\MatrixProduct{A}{B}})}\) as follows:
\begin{equation}\label{eq:TrilinearForm}
\Contraction{\tensor{S}}{}{3}:
\begin{array}[t]{ccc}
{\Field}^{\matrixsize{\firstdim}{\seconddim}}\times{{\Field}^{\matrixsize{\seconddim}{\thirddim}}}\times{({\Field}^{\matrixsize{\firstdim}{\thirddim}})}^{*}&\rightarrow&{\Field},\\
(\mat{A},\mat{B},\Transpose{C})&\mapsto&\langle\mat{C},\MatrixProduct{A}{B}\rangle.
\end{array}
\end{equation}
As the space of trilinear forms is the canonical dual space of order three tensor products, Strassen algorithm~(\ref{eq:StrassenMultiplicationAlgorithm}) is encoded as the tensor decomposition~\(\tensor{S}\) of the matrix multiplication tensor in sum of seven rank-one tensors defined by the following relations:
\begin{equation}\label{eq:StrassenTensor}
\begin{array}{r}
\tensor{S}=\sum_{i=1}^{7}{\LeftTensor{i}}\!\tensorproduct\!{\RightTensor{i}}\!\tensorproduct\!{\ProductTensor{i}}=
\mathcolor{\triadfive}{{\begin{smatrix}1&0\\0&1\end{smatrix}}\!\tensorproduct\!{\begin{smatrix}1&0\\0&1\end{smatrix}}\!\tensorproduct\!\begin{smatrix}1&0\\0&1\\\end{smatrix}}
\\[\bigskipamount]
+
\mathcolor{\triadfour}{\begin{smatrix}0&1\\0&-1\\\end{smatrix}\!\tensorproduct\!\begin{smatrix}0&0\\1&1\\\end{smatrix}\!\tensorproduct\!\begin{smatrix}1&0\\0&0\\\end{smatrix}}
\!+\!
\mathcolor{\triadseven}{\begin{smatrix}-1&0\\1&0\\\end{smatrix}\!\tensorproduct\!\begin{smatrix}1&1\\0&0\\\end{smatrix}\!\tensorproduct\!\begin{smatrix}0&0\\0&1\\\end{smatrix}}
\\[\bigskipamount]
+
\mathcolor{\triadtwo}{\begin{smatrix}1&1\\0&0\\\end{smatrix}\!\tensorproduct\!\begin{smatrix}0&0\\0&1\\\end{smatrix}\!\tensorproduct\!\begin{smatrix}-1&0\\1&0\\\end{smatrix}}
\!+\!
\mathcolor{\triadone}{\begin{smatrix}1&0\\0&0\\\end{smatrix}\!\tensorproduct\!\begin{smatrix}0&1\\0&-1\\\end{smatrix}\!\tensorproduct\!\begin{smatrix}0&0\\1&1\\\end{smatrix}}
\\[\bigskipamount]
+
\mathcolor{\triadsix}{\begin{smatrix}0&0\\0&1\\\end{smatrix}\!\tensorproduct\!\begin{smatrix}-1&0\\1&0\\\end{smatrix}\!\tensorproduct\!\begin{smatrix}1&1\\0&0\\\end{smatrix}}
\!+\!
\mathcolor{\triadthree}{\begin{smatrix}0&0\\1&1\\\end{smatrix}\!\tensorproduct\!\begin{smatrix}1&0\\0&0\\\end{smatrix}\!\tensorproduct\!\begin{smatrix}0&1\\0&-1\end{smatrix}}
\end{array}
\end{equation}
in~\({\BDual{{\Field}^{\matrixsize{\firstdim}{\seconddim}}}\tensorproduct{\BDual{{\Field}^{\matrixsize{\seconddim}{\thirddim}}}}\tensorproduct{{\Field}^{\matrixsize{\firstdim}{\thirddim}}}}\) with~\({{\firstdim}={\seconddim}={\thirddim}={2}}\).
In the above tensor decomposition, each summand is a \emph{rank-one tensor} and its \emph{tensor rank} is the number~\(r\) of such element~(\(7\) there).
Given \Cref{eq:StrassenTensor}, multiplication formula~(\ref{eq:mxnTimesnxp}) implemented by~\cref{eq:StrassenMultiplicationAlgorithm} is obtained using the third 2-contraction of the tensor~\({\tensor{\tensor{S}}\tensorproduct{\mat{A}}\tensorproduct{\mat{B}}}\) as defined in the following map:
\begin{equation}
\label{eq:SecondContraction}
\begin{array}{c}
{\left({\BDual{\Matrices{\Field}{\matrixsize{\firstdim}{\seconddim}}}}\!\tensorproduct{\BDual{\Matrices{\Field}{\matrixsize{\seconddim}{\thirddim}}}}\!\tensorproduct{\Matrices{\Field}{\matrixsize{\firstdim}{\thirddim}}}\right)}
\!\tensorproduct\!
{\left({\Matrices{\Field}{\matrixsize{\firstdim}{\seconddim}}}\!\tensorproduct\!{\Matrices{\Field}{\matrixsize{\seconddim}{\thirddim}}}\right)}
\!\rightarrow\!
{\Matrices{\Field}{\matrixsize{\firstdim}{\thirddim}}}, \\[\smallskipamount]
{\left(\sum_{i=1}^{r}{\LeftTensor{i}}\!\tensorproduct\!{\RightTensor{i}}\!\tensorproduct\!{\ProductTensor{i}}\right)}
\tensorproduct
({\mat{A}}\tensorproduct{\mat{B}})\mapsto
\sum_{i=1}^{r}\langle{\LeftTensor{i}},\mat{A}\rangle
\langle{\RightTensor{i}},\mat{B}\rangle
{\ProductTensor{i}}.
\end{array}
\end{equation}
Some formalisms are more adapted to the design of algorithms computing efficiently the matrix product (as shown in~\cref{sec:implem}) than direct tensor decompositions.
For example, a nice concise representation was introduced in~\cite{hopcroft:1973}; it encodes the sum of rank-one tensors by three matrices as done for the Strassen tensor decomposition~(\ref{eq:StrassenTensor}) in the following three matrices~\({\mat{L}_{\tensor{S}},\mat{R}_{\tensor{S}}}\) and~\(\mat{P}_{\tensor{S}}\):
\begin{equation}\label{eq:StrassenHMRepresentation}
\begin{smatrix}%
\mathcolor{\triadfive}{1}&\mathcolor{\triadfive}{0}&\mathcolor{\triadfive}{0}& \mathcolor{\triadfive}{1}\\
\mathcolor{\triadfour}{0}&\mathcolor{\triadfour}{1}&\mathcolor{\triadfour}{0}&\mathcolor{\triadfour}{{-1}}\\
\mathcolor{\triadseven}{-1}&\mathcolor{\triadseven}{0}&\mathcolor{\triadseven}{1}& \mathcolor{\triadseven}{0}\\
 \mathcolor{\triadtwo}{1}&\mathcolor{\triadtwo}{1}&\mathcolor{\triadtwo}{0}& \mathcolor{\triadtwo}{0}\\
 \mathcolor{\triadone}{1}&\mathcolor{\triadone}{0}&\mathcolor{\triadone}{0}& \mathcolor{\triadone}{0}\\
 \mathcolor{\triadsix}{0}&\mathcolor{\triadsix}{0}&\mathcolor{\triadsix}{0}& \mathcolor{\triadsix}{1}\\
 \mathcolor{\triadthree}{0}&\mathcolor{\triadthree}{0}&\mathcolor{\triadthree}{1}& \mathcolor{\triadthree}{1}
\end{smatrix},\quad%
\begin{smatrix}%
 \mathcolor{\triadfive}{1}&\mathcolor{\triadfive}{0}&\mathcolor{\triadfive}{0}& \mathcolor{\triadfive}{1}\\
 \mathcolor{\triadfour}{0}&\mathcolor{\triadfour}{0}&\mathcolor{\triadfour}{1}& \mathcolor{\triadfour}{1}\\
 \mathcolor{\triadseven}{1}&\mathcolor{\triadseven}{1}&\mathcolor{\triadseven}{0}& \mathcolor{\triadseven}{0}\\
 \mathcolor{\triadtwo}{0}&\mathcolor{\triadtwo}{0}&\mathcolor{\triadtwo}{0}& \mathcolor{\triadtwo}{1}\\
 \mathcolor{\triadone}{0}&\mathcolor{\triadone}{1}&\mathcolor{\triadone}{0}&\mathcolor{\triadone}{-1}\\
\mathcolor{\triadsix}{-1}&\mathcolor{\triadsix}{0}&\mathcolor{\triadsix}{1}& \mathcolor{\triadsix}{0}\\
 \mathcolor{\triadthree}{1}&\mathcolor{\triadthree}{0}&\mathcolor{\triadthree}{0}& \mathcolor{\triadthree}{0}
\end{smatrix},\quad%
\Transpose{%
\begin{smatrix}%
 \mathcolor{\triadfive}{1}&\mathcolor{\triadfive}{0}&\mathcolor{\triadfive}{0}& \mathcolor{\triadfive}{1}\\
 \mathcolor{\triadfour}{1}&\mathcolor{\triadfour}{0}&\mathcolor{\triadfour}{0}& \mathcolor{\triadfour}{0}\\
 \mathcolor{\triadseven}{0}&\mathcolor{\triadseven}{0}&\mathcolor{\triadseven}{0}& \mathcolor{\triadseven}{1}\\
\mathcolor{\triadtwo}{-1}&\mathcolor{\triadtwo}{1}&\mathcolor{\triadtwo}{0}& \mathcolor{\triadtwo}{0}\\
 \mathcolor{\triadone}{0}&\mathcolor{\triadone}{1}&\mathcolor{\triadone}{0}& \mathcolor{\triadone}{1}\\
 \mathcolor{\triadsix}{1}&\mathcolor{\triadsix}{0}&\mathcolor{\triadsix}{1}& \mathcolor{\triadsix}{0}\\
 \mathcolor{\triadthree}{0}&\mathcolor{\triadthree}{0}&\mathcolor{\triadthree}{1}&\mathcolor{\triadthree}{-1}
\end{smatrix}
}.%
\end{equation}
\begin{notation}
Given an~\(\matrixsize{\firstdim}{\seconddim}\)-matrix~\(\mat{A}\), we denote by~\(\row{\mat{A}}{i}\) the~\(i\)th row and by~\(\vectorization{\mat{A}}\) the row-major vectorization of this matrix, i.e.\ the vector~\(v\) in~\(\RR^{\firstdim\seconddim}\) such that~\({v_{i\seconddim+j} = a_{i,j}}\).
We also denote by~\(\matr{v}{\firstdim,\seconddim}\) the reciprocal operation, building an~\(\matrixsize{\firstdim}{\seconddim}\) matrix from an~\(\firstdim\seconddim\)-dimensional vector.
Thus, the~\(i\)th line~\({\row{\mat{L}_{\tensor{S}}}{i}}\) (resp.~\({\row{\mat{R}_{\tensor{S}}}{i}}\)) of matrix~\(\mat{L}_{\tensor{S}}\) (resp.~\(\mat{R}_{\tensor{S}}\)) is the transposition of the row-major vectorization~\({\vectorization{\LeftTensor{i}}}\) of the first (resp.\ second~\({\vectorization{\RightTensor{i}}}\)) component of the~\(i\)th triad in~\Cref{eq:StrassenTensor} and the~\(i\)th column of matrix~\(\mat{P}_{\tensor{S}}\) is the column-major vectorization~\({\vectorization{\ProductTensor{i}}}\) of its third component.
\end{notation}
\begin{definition}\label{def:HMRepresentation}
This encoding of a tensor by three suitable matrices~\({\mat{L},\mat{R},\mat{P}}\) is called an \textsc{hm} representation and is denoted by~\(\HMRepresentation{\mat{L}}{\mat{R}}{\mat{P}}\).
\end{definition}
\Cref{eq:HMRepresentation2MatrixMultiplicationFormula} presented in~\cref{sec:accuracyTheoriticalBound} shows that the \textsc{hm} representation allows constructing \textsc{slp}s for the associated algorithms.
We show in~\cref{ssec:factor} that this could be done efficiently, e.g.\ using the kernel of~\(\mat{L}\) (resp.~\(\mat{R}\)) and Tellegen's transposition applied to~\(\mat{P}\).
Now we turn to symmetries of matrix product tensor decomposition.
Remark that the matrix product is associated to~\({\Trace(\MatrixProduct{\mat{A}}{\MatrixProduct{\mat{B}}{\mat{C}}})}\) by~\Cref{eq:TrilinearForm} and that, given invertible matrices~\({\mat{U},\mat{V},\mat{W}}\) of suitable sizes and the classical trace properties, this trace is equal to:
\begin{equation}\label{eq:isotropy}
\begin{array}{l}
\Trace\bigl(\Transpose{(\MatrixProduct{\MatrixProduct{\mat{A}}{\mat{B}}}{\mat{C}})}\bigr)
=\Trace(\MatrixProduct{\mat{C}}{\MatrixProduct{\mat{A}}{\mat{B}}})
=\Trace(\MatrixProduct{\mat{B}}{\MatrixProduct{\mat{C}}{\mat{A}}})\\
\textrm{and to}\ \Trace\bigl(\MatrixProduct{\Inverse{\mat{U}}}{\MatrixProduct{\mat{A}}{\mat{V}}}
\cdot\Inverse{\mat{V}}\cdot{\mat{B}}\cdot{\mat{W}}\cdot\Inverse{\mat{W}}\cdot{\mat{C}}\cdot{\mat{U}}\bigr).
\end{array}
\end{equation}
These relations illustrate the next theorem and induce the isotropy action on matrix product tensor decomposition presented below:
\begin{theorem}[{\cite[\S~2.8]{groot:1978a}}]\label{thm:groot}
The isotropy group of the~\(\matrixsize{\firstdim}{\firstdim}\) matrix multiplication tensor is the semidirect product~\({{{\textsc{psl}^{\pm}({{\Field}^{\firstdim}})}{}^{\times{3}}}\!\rtimes{\mathfrak{S}_{3}}}\), where~\(\textsc{psl}\) stands for the group of matrices of determinant~\({\pm{1}}\) and~\(\mathfrak{S}_{3}\) for the symmetric group on~\(3\) elements.
\end{theorem}
\begin{definition}\label{lem:sandwiching}
Let~\(\Isotropy{g}\) denotes~\({(\mat{U}\times\mat{V}\times\mat{W})}\) in~\({{\textsc{psl}^{\pm}({\Field}^{\firstdim})}{}^{\times{3}}}\) and~\(\tensor{T}\) a rank-one tensor~\({\mat{A}\tensorproduct\mat{B}\tensorproduct\mat{C}}\); the action~\({\IsotropyAction{\Isotropy{g}}{\tensor{T}}}\) of~\(\Isotropy{g}\) on~\(\tensor{T}\) is the rank-one tensor~\({
{\left(\MatrixProduct{\InvTranspose{\mat{U}}}{\MatrixProduct{\mat{A}}{\Transpose{\mat{V}}}}\right)}\!
\tensorproduct\!
{\left(\MatrixProduct{\InvTranspose{\mat{V}}}{\MatrixProduct{\mat{B}}{\Transpose{\mat{W}}}}\right)}\!
\tensorproduct\!
{\left(\MatrixProduct{\InvTranspose{\mat{W}}}{\MatrixProduct{\mat{C}}{\Transpose{\mat{U}}}}\right)}}\).
This action is extended by additivity to higher tensor rank tensors.
\par
Given two isotropies~\(g_{1}\) defined by matrices~\({({\mat{U}_{1}}\times{\mat{V}_{1}}\times{\mat{W}_{1}})}\) and~\({g_{2}}\) defined by matrices~\({({\mat{U}_{2}}\times{\mat{V}_{2}}\times{\mat{W}_{2}})}\) both in~\({\textsc{psl}^{\pm}({{\Field}^{\firstdim}})}{}^{\times{3}}\), the composition~\({g_{1}\IsotropyComposition g_{2}}\) is given by~\({(\MatrixProduct{\mat{U}_{1}}{\mat{U}_{2}}\times{\MatrixProduct{\mat{V}_{1}}{\mat{V}_{2}}}\times{\MatrixProduct{\mat{W}_{1}}{\mat{W}_{2}}})}\).
\end{definition}
The isotropies action on an \textsc{hm} representation is a direct consequence of the above results and presented in the following lemma.
\begin{lemma}\label{lem:actionOnHMRepresentation}
Let~\(\Isotropy{g}\) be~\({{({\mat{U}}\times{\mat{V}}\times{\mat{W}})}}\) in~\({{\textsc{psl}^{\pm}({\Field}^{\firstdim})}^{\times{3}}}\) and~\(\HMRepresentation{\mat{L}}{\mat{R}}{\mat{P}}\) be an \textsc{hm} representation of a matrix product tensor decomposition, the action~\({\IsotropyAction{\Isotropy{g}}{\HMRepresentation{\mat{L}}{\mat{R}}{\mat{P}}}}\) of~\(\Isotropy{g}\) on~\(\HMRepresentation{\mat{L}}{\mat{R}}{\mat{P}}\) is another \textsc{hm} representation of a matrix product tensor decomposition defined by:
\begin{equation}\label{eq:isotropyActionOnHMRepresentation}
\HMRepresentation%
{\MatrixProduct{L}{\bigl({\Transpose{\mat{V}}\tensorproduct{\Inverse{\mat{U}}}}\bigr)}}%
{\MatrixProduct{R}{\bigl({\Transpose{\mat{W}}\tensorproduct{\Inverse{\mat{V}}}}\bigr)}}%
{\MatrixProduct{\bigl({{\mat{U}}\tensorproduct{\InvTranspose{\mat{W}}}}\bigr)}{P}}.
\end{equation}
\end{lemma}
Dealing with a tensor decomposition or with the associated \textsc{hm} representation is not strictly equivalent; In~\Cref{lem:sandwiching} there is no need to care about the determinants of the matrices~\({(\mat{U},\mat{V},\mat{W})}\) while this fact is no more true for \Cref{eq:isotropyActionOnHMRepresentation} as (say)~\(\mat{U}\) acts on two different components.
\par
The following theorem recalls that all~\(\matrixsize{2}{2}\)-matrix product algorithms with~\(7\) coefficient multiplications are obtained by this single orbit of the action of isotropies on Strassen tensor decomposition:
\begin{theorem}[{\cite[\S~0.1]{groot:1978}}]\label{thm:IsotropiesActTransitivelyOnOptimalAlgorithm}
The group~\({{\textsc{psl}^{\pm}({{\Field}^{\firstdim}})}^{\times{3}}}\) acts transitively on the variety of fast algorithms multiplying~\(\matrixsize{2}{2}\)-matrices.
\end{theorem}
Thus, isotropy action on Strassen tensor decomposition may define other matrix product algorithm of same tensor rank but with potentially more interesting characteristics as shown in \Cref{sec:numericalStabilityMeasure}.
We make explicit these properties in the following section.
\section{Bilinear operator accuracy bound}\label{sec:accuracyTheoriticalBound}
We will consider that any finite-dimensional real vector space~\(\mathbb{U}\) is equipped with a norm~\(\norm{\cdot}\) and denote by~\(\dualnorm{\cdot}\) the related dual norm; for~\({\phi:\mathbb{U} \rightarrow \RR}\), its norm~\(\dualnorm{\phi}\) is~\({\text{sup}(|\phi(v)|,\|v\|\leq{1})}\).
For instance, the max-norm~\(\maxnorm{\cdot}\) and the one-norm~\(\onenorm{\cdot}\) are dual one with the other, while the two-norm~\(\twonorm{\cdot}\) is self-dual.
We will also denote the Hamming weight~\({\#\{i|x_i\neq 0\}}\) of~\(x\) by~\(\zeronorm{x}\).
The~\(n\)-dimensional vector of coefficients~\({x_1,\dots,x_n}\) is denoted by~\({{(x_i)}_{i\in\{1..n\}}}\) or more succinctly~\(\vectorif{x_i}{i}\) when the indexing is clear from the context.%
By extension, we denote~\({\norm{x}\norm{y}}\) by~\({\norm{x;y}}\) and~\({\norm{\mat{L}}\norm{\mat{R}}\norm{\mat{P}}}\) by~\({\norm{{\mat{L}};{\mat{R}};{\mat{P}}}}\).
\par
\begin{lemma}\label{lem:operatornorms}
For any matrix~\(\mat{A}\) in~\(\RR^{\matrixsize{\firstdim}{\seconddim}}\) and any vectors~\({x,y}\) in~\(\RR^{\seconddim}\) the following inequalities hold:
\begin{align}\label{eq:matvecmaxnorm}
	{|x \cdot y|} &\leq \dualnorm{x}\norm{y},\quad\quad\quad\quad\norm{\mat{A}x}\leq\norm{\vectorif{\dualnorm{\row{\mat{A}}{i}}}{i}}\norm{x},\\\label{eq:matvecop}
\textmaxnorm{\mat{A}x}&\leq\max_{i=1\dots{\firstdim}}\left({\textstyle\sum_{j=1}^{\seconddim}|a_{i,j}|}\right)\textmaxnorm{x}\leq{\seconddim\textmaxnorm{\mat{A}}\textmaxnorm{x}}.
\end{align}
\end{lemma}
Given an \textsc{hm}
representation~\(\HMRepresentation{\mat{L}}{\mat{R}}{\mat{P}}\) of a
matrix multiplication tensor decomposition, one can retrieve the
transpose of the multiplication formula~(\ref{eq:mxnTimesnxp})
implemented by~\cref{eq:StrassenMultiplicationAlgorithm} using the
Hadamard product~\(\HadamardProduct{\mat{A}}{\mat{B}}\) of
matrices~\(\mat{A}\) and~\(\mat{B}\) with the following map:
\begin{equation}\label{eq:HMRepresentation2MatrixMultiplicationFormula}
\begin{array}{ccc}
{{\Matrices{\Field}{\matrixsize{\firstdim}{\seconddim}}}\times{\Matrices{\Field}{\matrixsize{\seconddim}{\thirddim}}}}
&\rightarrow&
{\Matrices{\Field}{\matrixsize{\firstdim\thirddim}{1}}},\\
({\mat{A}},{\mat{B}})&\mapsto&\MatrixProduct%
{\Transpose{\mat{P}}}%
{\Bigl(\HadamardProduct%
{\bigl(\MatrixProduct{L}{\vectorization{\mat{A}}}\bigr)}%
{\bigl(\MatrixProduct{R}{\vectorization{\mat{B}}}\bigr)}%
\Bigr).}%
\end{array}
\end{equation}
Hence, we express there a bilinear operator~\({\beta : {\RR^{e}} \times {\RR^{f}} \rightarrow {\RR^{g}}}\) represented by its \textsc{hm} representation~\({\HMRepresentation{\mat{L}}{\mat{R}}{\mat{P}}}\) in~\({\RR^{r\times e}\times \RR^{r\times f}\times \RR^{r\times g}}\) as~\({\beta(u,v)=\sum_{i=1}^r(\row{\mat{L}}{i}\cdot{u})(\row{\mat{R}}{i}\cdot{v})\row{(\Transpose{\mat{P}})}{i}}\).
When this operator encodes an~\(\matrixsize{\firstdim}{\seconddim}\) by~\(\matrixsize{\seconddim}{\thirddim}\) matrix multiplication formula, we will thus denote it by~\(\beta_{\textsc{mm}}\) and we will have~\({e=\firstdim\seconddim,f=\seconddim\thirddim,g=\firstdim\thirddim}\).
We also consider recursive applications of such operators defined as:
\begin{equation}
\beta^{(\ell)}:
\begin{array}[t]{cl}
\RR^{e_0e^\ell}\times\RR^{f_0f^\ell}&\rightarrow\RR^{g_0g^\ell},\\
(u,v) &\mapsto \sum_{i=1}^r \beta^{(\ell-1)}(\row{\mat{L}}{i} \cdot u, \row{\mat{R}}{i}\cdot v) \row{(\Transpose{\mat{P}})}{i}
\end{array}
\end{equation}
and~\({\beta^{0}:\RR^{e_0}\times \RR^{f_0}\rightarrow \RR^{g_0}}\), a bilinear operator which we will assume to be bounded:~\({\bigl\|\beta^{(0)}(u,v)\bigr\| \leq \gamma_{0} \|u\|\|v\|}\) for all~\({(u,v)}\) in~\({\RR^{e_0}\times\RR^{f_0}}\).
For convenience, we will define the dimensions~\(G\) as~\({{{g_{0}}{g^{\ell}}}}\) and~\(K\) as~\({{k_{0}}{k^{\ell}}}\).
Recall that~\(\row{(\Transpose{\mat{P}})}{i}\) is the~\(i\)th column of~\(\mat{P}\) and remark that the expression~\({\row{\mat{L}}{i}\cdot{u}}\) is an abuse of notation for the operation where each coefficient~\(l_{i,j}\) of~\(\mat{L}\) multiplies a block of~\(e_0e^{\ell-1}\) contiguous coefficients of~\(u\), namely:~\({\row{\mat{L}}{i}\cdot{u}=\Transpose{(\row{\mat{L}}{i}\matr{u}{e,e_0e^{\ell-1}})}}\).
We will consider the floating point arithmetic in the standard model of~\cite{Higham:2002}: \(\comp{x}\) denotes the computed value for an expression~\(x\) such that~\({\comp{a \textup{ op } b} = (a \textup{ op } b)(1+\delta)}\) for~\({\textup{op}=+,-,\times,/}\) where~\(\delta\) is the unit round off such that~\({|\delta|\leq \ulp}\), except when~\({a \text{ op }b}\) is~\(0\) where~\(\delta\) is~\(-1\).
\par
We recall in the following Lemma some classical inequalities:
\begin{lemma}[see~{\cite[Eq.~(3.5)]{Dai:2023aa}} and~{\cite[Eq.~(4.4)]{Higham:2002}}]\label{lem:dotprodbound}
For any vectors~\(u\) and~\(v\) in~\(\RR^{\thirddim}\) the following inequalities hold:
\begin{eqnarray}
|\comp{u \cdot v}-u\cdot v| &\leq& \zeronorm{u}\dualnorm{u} \|v\|\ulp+\bbigO{\ulp^2},\\\label{eq:dotprodbound}
\left|\comp{\textstyle\sum_{i=1}^n u_i}-\textstyle\sum_{i=1}^n u_i\right| &\leq& (n-1)\Bigl({\textstyle\sum_{i=1}^{n} |u_i|}\Bigr)\ulp +\bbigO{\ulp^2}.\label{eq:errorForSum}
\end{eqnarray}
\end{lemma}
We define now the growth factor used in this work.
\begin{definition}\label{def:gf}
The \textit{growth factor}~\(\gamma\) of the formula~\(\HMRepresentation{\mat{L}}{\mat{R}}{\mat{P}}\) computing the bilinear form~\(\beta\) is defined by~\({\displaystyle\max_{j=1\dots g} {\scriptstyle\sum_{i=1}^r} \dualnorm{\row{\mat{L}}{i}} \dualnorm{\row{\mat{R}}{i}} |p_{i,j}|}\).
\end{definition}
\par
The growth factor not only bounds the values of bilinear operators, as show in~\cref{lem:valuebound}, but is also central in analyzing their forward numerical error, which will be the focus of~\cref{th:recbound}.
\begin{lemma}\label{lem:valuebound}
For any~\(u,v\) with adequate dimensions, the following relations hold:
\({{\norm{\beta(u,v)}}\leq{\gamma\norm{u}\norm{v}}}\), \({{\textbignorm{\beta^{(\ell)}(u,v)}}\leq{\gamma_{0}\gamma^{\ell}\norm{u}\norm{v}}}\) and~\({{\textbigmaxnorm{\beta_\textsc{mm}^{(\ell)}(u,v)}}\leq{k_{0}k^{\ell}\textmaxnorm{u}\textmaxnorm{v}}}\).
\end{lemma}
\begin{proof}
Let~\(\mat{D_{j}}\) denotes~\({\text{Diag}_{i=1\dots r}(p_{i,j})}\)
and~\(c_j\) be the~\(j\)th coefficient of~\(\beta(u,v)\).
We have that
\({|c_j|\leq\norm{\Transpose{u} \Transpose{L} \mat{D}_j\mat{R}
 \mat{v}}\!\leq\!
\dualnorm{\Transpose{u}\mat{L}\mat{D}_j\mat{R}}\!\norm{v}}\),
so that
\(|c_j|\leq\dualnorm{\Transpose{L}\mat{D}_j\mat{R}}\norm{u}\norm{v}\leq
\dualnorm{\sum_{i=1}^r (\row{\mat{L}}{i}\otimes\row{\mat{R}}{i})p_{i,j}}\norm{u}\norm{v}\)
and
\(|c_j|\leq\left(\sum_{i=1}^r\dualnorm{\row{\mat{L}}{i}}\dualnorm{\row{\mat{R}}{i}}|p_{i,j}|\right)\norm{u}\norm{v}\).
Finally, the last inequality follows from~\eqref{eq:matvecmaxnorm}.
\end{proof}
\begin{theorem}\label{th:recbound}
Given any choice of norm~\(\norm{\cdot}\), if~\(G\) denotes~\({g_0g^\ell}\) and~\(K\) denotes~\({k_0k^\ell}\), the error in computing~\(\beta^{(\ell)}\) is bounded as follows~\({\textbigmaxnorm{\widehat{\beta^{(\ell)}(u,v)}-\beta^{(\ell)}(u,v)}\leq\kappa\norm{u}\norm{v}\ulp+\bbigO{\ulp^2}}\) where either
\begin{equation}\label{eq:mmaccuracy}
\kappa={\left(\frac{K}{k_0}\right)}^{\log_{k}\gamma}\left({k_0}^{\!2} + \frac{{{Q}_0}{k_0}\gamma}{(\gamma-k)}\right)-\frac{{Q_0}{K}\gamma}{(\gamma-k)}
\end{equation}
when~\(\beta^{(\ell)}\) is an~\({{M}\times{K}}\) by~\({{K}\times{N}}\) matrix multiplication, or
\begin{equation}\label{eq:opaccuracy}
\kappa = {\left(\lfrac{G}{g_0}\right)}^{\log_{g}\gamma} \gamma_0\left(1+\bigl({1+\log_g(\lfrac{G}{g_0})}\bigr)Q_0\right),
\end{equation}
otherwise, and~\(Q_0={\max_j \bigl(\zeronorm{\row{(\Transpose{P})}{j}}\!+ \max_i(\zeronorm{\row{\mat{L}}{i}}\!+\zeronorm{\row{\mat{R}}{i}}) \mathbb{1}_{p_{i,j}\neq0}\bigr)}\) as in~\cite[Definition~1]{BBDLS16}.
\end{theorem}
\begin{table*}[tbp]
\begin{center}\setlength{\tabcolsep}{4pt}\renewcommand{\arraystretch}{.9}
\begin{tabular}{lccccc|cc|cc|cc|cc@{}}
\toprule
&Formula&Applies to&norm&&&\multicolumn{2}{c}{Winograd}& \multicolumn{2}{c}{Strassen}&\multicolumn{2}{c}{\Cref{eq:powers}}&\multicolumn{2}{c}{\Cref{eq:asopt}}\\
&&&&\(Q\)&\(\gamma\)&\(Q\)&\(\gamma\)&\(Q\)&\(\gamma\)&\(Q\)&\(\gamma\)&\(Q\)&\(\gamma\)\\
\midrule
Brent~\cite{brent:1970a}&\eqref{eq:mmaccuracy}&Strassen only&\(\infty\)&\textsc{na}&\(\gamma_{1,1,\infty}\)&&&3.67&12&&&\\
BL~\cite{bini:1980} DDHK~\cite{demmel:2007a}&\eqref{eq:opaccuracy}&any \textsc{mm} alg.&\(\infty\)&\(Q'_0\)\cref{fn:balancedadd}&\(\gamma_{0,1,\infty}\)\cref{fn:binibug}&9&18&7&12&9.59&40&9.81&98.54\\
Higham~\cite{Higham:2002}&\eqref{eq:mmaccuracy}&S.\ \& W.\ only&\(\infty\)&\textsc{na}&\(\gamma_{1,1,\infty}\)&4.94&18&3.83&12&&&\\
Ballard et al.~\cite{BBDLS16}&\eqref{eq:opaccuracy}&any \textsc{mm} alg.&\(\infty\)&\(Q_0\)&\(\gamma_{1,1,\infty}\)&10&18&8&12&12&13&15&17.48\\
\multirow{2}{*}{Dai, Lim~\cite{Dai:2023aa}}&\multirow{2}{*}{\cite[Th~3.3]{Dai:2023aa}}&\multirow{2}{*}{\({\ell=1}\), any alg.}&\(2\)&\({m+n+r}\) &\(\gamma_{2,1}\)&15&17.86&15&14.83&15&12.21&15&12.07\\
&&&\(\infty\)&\({m+n+r}\)&\(\gamma_{1,\infty,1}\)&15&27&15&20&15&22&15&25.14\\
\multirow{2}{*}{Here}&\multirow{2}{*}{\eqref{eq:mmaccuracy}}&\multirow{2}{*}{any \textsc{mm} alg.}&\(2\)&\(Q_{0}\)&\(\gamma_{2,1,\infty}\)&10&8&8&6.83&12&6.05&15&5.97\\
&&&\(\infty\)&\(Q_{0}\)&\(\gamma_{1,1,\infty}\)&10&18&8&12&12&13&15&17.48\\
Here&\eqref{eq:opaccuracy}&any alg.&\(\infty\)&\(Q_{0}\)&\(\gamma_{1,1,\infty}\)&10&18&8&12&12&13&15&17.48\\
\bottomrule
\end{tabular}
\caption{Comparing accuracy   formulas %
for recursive bilinear matrix multiplication operators in the form of~\cref{th:recbound}.
}\label{tab:bounds:sota}
\end{center}
\end{table*}
\begin{proof}
By induction, we will prove that the bound is of the form
\(%
{{\textbigmaxnorm{\Delta_{\beta^{(\ell)}}}} ={\textbigmaxnorm{\widehat{\beta^{(\ell)}(u,v)} - \beta^{(\ell)}(u,v)}} \leq{t_{\ell} \norm{u}\norm{v}\ulp +\bbigO{\ulp^2}}}\),
clarifying in the process the value for~\(t_{\ell}\).
Consider the block~\(c_j\) of~\({\lfrac{G}{g}=g_{0}g^{\ell-1}}\) consecutive output coefficients:~\({c_j=\sum_{i=1}^r \mat{H}_i p_{i,j}}\), where~\({\mat{H}_i={\beta^{(\ell-1)}(\row{\mat{L}}{i}\cdot{u},\row{\mat{R}}{i}\cdot{v})}}\).%
Consider definitions~\({d_{i,j}={\mat{H}}_{i}p_{i,j}}\) and~\({\Delta_{d_{i,j}}=\comp{d_{i,j}}-d_{i,j}}\).
Then, by~\cref{eq:errorForSum}:
\begin{align}
	\textmaxnorm{\comp{c_j}-c_j}&\leq\maxnorm{\comp{\sum_{i=1}^r\comp{d_{i,j}}}{-}\sum_{i=1}^r \comp{d_{i,j}}}+\maxnorm{\sum_{i=1}^r\comp{d_{i,j}}- \sum_{i=1}^r d_{i,j}},\\
	\omit\rlap{${\leq}\sum_{i=1}^{r} \maxnorm{\comp{{\mat{H}}_{i}p_{i,j}}}\bigl(\zeronorm{\row{(\Transpose{P})}{j}}-1\bigr)\ulp+\sum_{i=1}^r \textmaxnorm{\Delta_{d_{i,j}}} +\bbigO{\ulp^2}.\label{eq:deltaci}$}\\
	\maxnorm{\Delta_{d_{i,j}}}&\leq\maxnorm{\comp{p_{i,j}\comp{\mat{\mat{H}_i}}}-p_{i,j}\comp{\mat{H}_i}}+\maxnorm{p_{i,j}\comp{\mat{H}_i}-p_{i,j}\mat{H}_i},\\
	&\leq |p_{i,j}|\textmaxnorm{\mat{H}_i}\ulp + |p_{i,j}|\textmaxnorm{\Delta_{\mat{H}_i}} +\bbigO{\ulp^2}.\label{eq:deltadij}
\end{align}
	\(\textmaxnorm{\Delta_{\mat{H}_i}}\) is equal to~\({\textbigmaxnorm{\comp{\beta^{(\ell-1)}}\bigl({\comp{\row{\mat{L}}{i}\cdot{u}},\comp{\row{\mat{R}}{i}\cdot{v}}}\bigr)-\beta^{(\ell-1)}(\row{\mat{L}}{i}\cdot{u},\row{\mat{R}}{i}\cdot{v})}}\) by bilinearity of~\(\beta^{(\ell-1)}\) and bounded by:
\begin{equation}\label{eq:deltaQ}
\maxnorm{\Delta_{\beta^{(\ell-1)}}}+ \maxnorm{\beta^{(\ell-1)}(\Delta_L,\row{\mat{R}}{i}\cdot{v})}+\maxnorm{\beta^{(\ell-1)}(\row{\mat{L}}{i}\cdot{u},\Delta_{R})}.
\end{equation}
By~\cref{eq:dotprodbound} and the induction hypothesis we have
\begin{align}
\maxnorm{\Delta_\mat{M}} &\leq  \zeronorm{\row{\mat{M}}{i}}\dualnorm{\row{\mat{M}}{i}}\norm{u}\ulp +\bbigO{\ulp^2}\ \textup{with}\ \mat{M}\in \{\mat{L},\mat{R}\}, \label{eq:deltaM} \\
\maxnorm{\Delta_{\beta^{(\ell-1)}}} &\leq t_{\ell-1} \norm{\row{\mat{L}}{i}\cdot u+\Delta_\mat{L}}\norm{\row{\mat{R}}{i}\cdot v+\Delta_\mat{R}}\ulp +\bbigO{\ulp^2}, \\
  &\leq c_{\ell -1} \dualnorm{\row{\mat{L}}{i}} \norm{u} \dualnorm{\row{\mat{R}}{i}}\norm{v}\ulp + \bbigO{\ulp^2}.\label{eq:betalm1}
\end{align}
By~\cref{lem:valuebound}, the following inequality holds
\begin{equation}\label{eq:betadeltaL}
  \maxnorm{\beta^{(\ell-1)}(\Delta_\mat{L},\row{\mat{R}}{i}\cdot v)}
  \leq \Theta_{0} \Theta^{\ell-1} \textmaxnorm{\Delta_\mat{L}}\dualnorm{\row{\mat{R}}{i}}\norm{v}
\end{equation}
for~\((\Theta,\Theta_0)
=(k,k_0)\) if~\(\beta=\beta_{\textsc{mm}}\) or~\({(\gamma,\gamma_0)}\) otherwise (where
\(\Theta_0=\gamma_0\) comes from the current proof with~\({\ell=1}\) and~\({g_0=1}\)).
Similarly,
\begin{equation}\label{eq:betadeltaR}
\maxnorm{\beta^{(\ell-1)}(\row{\mat{L}}{i}\cdot{u},\Delta_\mat{R})}\leq\Theta_{0}\Theta^{\ell-1}\textmaxnorm{\Delta_\mat{R}}\dualnorm{\row{\mat{L}}{i}}\norm{u}.
\end{equation}
Gathering~\cref{eq:deltaci,eq:deltadij,eq:deltaQ,eq:deltaM,eq:betalm1,eq:betadeltaL,eq:betadeltaR} we deduce that
\begin{multline}
\maxnorm{\comp{c_j}-c_j}\leq\textstyle\sum_{i=1}^{r} \Bigl(\Theta_{0}\Theta^{\ell-1}\bigl(\zeronorm{\row{\mat{L}}{i}}+\zeronorm{\row{\mat{R}}{i}}+\zeronorm{\row{(\Transpose{P})}{j}}\bigr)+t_{\ell-1}\Bigr)\\
\times\dualnorm{\row{\mat{L}}{i}}\dualnorm{\row{\mat{R}}{i}}|p_{i,j}|\norm{u}\norm{v}\ulp+\bbigO{\ulp^{2}},
\end{multline}
and thus that~\({\textmaxnorm{\comp{c_j}-c_j}\leq\bigl({\Theta_{0}\Theta^{\ell-1}Q_{0}+t_{\ell-1}}\bigr)\gamma\norm{u}\norm{v}\ulp+\bbigO{\ulp^2}}\).
As in~\cite{Higham:2002}, we deduce that~\(t_\ell\) must then satisfy:
\begin{equation}
\left\{\begin{array}{llll}
t_{\ell} &=& \bigl({ \Theta_{0} \Theta^{\ell-1} Q_{0}+t_{\ell-1}}\bigr) \gamma &\text{for}\ \ell>0,\\
t_{0}  &=& {k_{0}}^{\!2} & \text{for matrix product,}\\
t_{0}   &=& (1+Q_{0})\gamma_{0} & \text{otherwise}.%
\end{array}\right.
\end{equation}
This recurrence relation solves into~\({t_\ell =  \gamma^{\ell}t_0\!+ Q_0\Theta_0\Theta^\ell\sum_{i=1}^{\ell} {\left(\lfrac{\gamma}{\Theta}\right)}^i}\).
\par
In the case of a matrix multiplication operator,~\(t_\ell\) is equal to:
\begin{equation}
\gamma^{\ell} {k_{0}}^{\!2} + Q_0k_0 \gamma\frac{ \gamma^{\ell}-k^{\ell}}{\gamma-k}
= {\left(\frac{K}{k_0}\right)}^{\log_k \gamma}\left({{k_{0}}^{\!2} + \frac{Q_0\gamma}{\gamma-k} k_0}\right)-\frac{Q_0\gamma}{\gamma-k}K.
\end{equation}
In the general case, the value of~\(t_\ell\) becomes:
\({(1 + (1+\ell)Q_{0})\gamma_0\gamma^{\ell}}\), that is equal to~\({{\left(\lfrac{G}{g_0}\right)}^{\log_{g}\gamma}\gamma_0\Bigl({1+\bigl({1+\log_g(\lfrac{G}{g_0})}\bigr)Q_0}\Bigr)}\).
\end{proof}
\footnotetext[1]{\label{fn:balancedadd}\cite{bini:1980,demmel:2007a} reach an improved value of~\(Q_0\) by assuming all additions are performed following a balanced tree, instead of a worst case estimate as done in all other formulas.}
\footnotetext[2]{\label{fn:binibug}We applied the same~\(\gamma_{0,1,\infty}\) for~\cite{bini:1980}, as it seems to be missing a dependency in the magnitude of the coefficients in~\({\mat{L},\mat{R},\mat{P}}\), which was fixed in~\cite{demmel:2007a}.}

\cref{th:recbound} generalizes or improves on previous similar results in~\cite{brent:1970a,Higham:2002,demmel:2007a,BBDLS16,Dai:2023aa}.
In fact,~\cite{Dai:2023aa} considers a single recursive level without base case;~\cite{brent:1970a,Higham:2002} have tight bounds but only for Strassen and Winograd's algorithms in max-norm;
lastly~\mbox{\cite{demmel:2007a,BBDLS16}} has an additional logarithmic factor likely due to a looser bound on each~\(\textbigmaxnorm{\beta^{(\ell-1)}}\), not exploiting the fact that they are matrix products.
\par
Even though the choice of the max-norm produces the tightest bounds in~\cref{th:recbound}, as in most previous works, the bounds are stated there for any choice of norm, as in~\cite{Dai:2023aa}.
Alternative norms, such as the~\(2\)-norm, may give growth factor expressions more amenable to optimizations, as detailed in \Cref{sec:numericalStabilityMeasure}.
\par
\cref{tab:bounds:sota} compares the various existing bounds on numerical accuracy of matrix multiplication algorithms.
They depend on the following choices made on the norms to define the growth factor~\(\gamma\):
\begin{equation}\begin{split}
	\gamma_{0,1,\infty}
	&= \maxnorm{\vectorif{\onenorm{\vectorif{\zeronorm{\row{\mat{L}}{i};\row{\mat{R}}{i}}|p_{i,k}|}{i}}}{k}} \|\mat{L}\|_\infty\|\mat{R}\|_\infty\|\mat{P}\|_\infty,\\
	&= \left(\max_{k\in\{1\ldots mn\}}\!
	\textstyle\sum_{i=1}^r\!\|\row{\mat{L}}{i}\|_0\|\row{\mat{R}}{i}\|_0|p_{i,k}|\right)\!\|\mat{L}\|_\infty\|\mat{R}\|_\infty\|\mat{P}\|_\infty.
\end{split}\end{equation}
\begin{equation}\label{def:growthfactor}
	\gamma_{2,1}
	=
	\onenorm{\vectorif{\twonorm{\row{\mat{L}}{i};\row{\mat{R}}{i};\row{\mat{P}}{i}}}{i}}
	=
	\textstyle\sum_{i=1}^r{\|\row{\mat{L}}{i}\|}_{2}{\|\row{\mat{R}}{i}\|}_{2}{\|\row{\mat{P}}{i}\|}_{2}.
	\hspace{25pt}
\end{equation}
\begin{equation}\label{eq:gfq1inf}
	\begin{split}\gamma_{q,1,\infty}
		&= \maxnorm{\vectorif{\onenorm{\vectorif{\xnorm{\row{\mat{L}}{i};\row{\mat{R}}{i}}{q}|p_{i,k}|}{i}}}{k}},\\
		&= \max_{k\in\{1\ldots mn\}} \textstyle\sum_{i=1}^r\|\row{\mat{L}}{i}\|_q\|\row{\mat{R}}{i}\|_q|p_{i,k}|\ \textup{with }q\in\{1,2\}.
	\end{split}
\end{equation}
\section{Growth factor along orbits}\label{sec:numericalStabilityMeasure}
In the footstep of~\cite{bini:1980}, we aim to find an alternative~\(\matrixsize{2}{2}\) matrix product tensor decomposition, in the orbit of Strassen's one, with improved accuracy, hence minimizing the growth factor.
The use of the maxnorm induces an expression~\cref{eq:gfq1inf} for~\(\gamma_{1,1,\infty}\) poorly suited for optimizations.
We will instead make two relaxations: first, using the~\(2\)-norm and second, as in~\cite{Dai:2023aa}, bounding~\(\gamma_{2,1,\infty}\) by~\(\gamma_{2,1}\):
\begin{equation}
\max_{k\in\{1\ldots mn\}}\sum_{i=1}^r\twonorm{\row{\mat{L}}{i}}\twonorm{\row{\mat{R}}{i}}|p_{i,k}| \leq \sum_{i=1}^r \twonorm{\row{\mat{L}}{i}}\twonorm{\row{\mat{R}}{i}}\twonorm{\row{\mat{P}}{i}}.
\end{equation}
\Cref{thm:IsotropiesActTransitivelyOnOptimalAlgorithm} shows that all fast~\(\matrixsize{2}{2}\) matrix product algorithms are in the same orbit under isotropies action introduced in~\cref{lem:sandwiching}.
While the tensor rank is invariant under this action, the growth factor is generally not.
As its definition is based on Frobenius norm, some isotropies leave it invariant as stated in the following lemma:
\begin{lemma}\label{lem:actionLeavingGrowthFactorInvariant}
The growth factor~\(\gamma_{2,1}\) is invariant under the action of the semidirect product~${{{\textup{\textsc{so}}^{\pm}({\Field^{n}})}{}^{\times{3}}}\!\rtimes{\mathfrak{S}_{3}}}$ induced by the special orthogonal group and the permutation group~\(\mathfrak{S}_{3}\).
\end{lemma}
\begin{proof}
By~\cref{def:FrobeniusInnerProduct}, Frobenius norms are invariant under orthogonal transformations and so is~\(\gamma_{2,1}\) by \cref{def:growthfactor}.
\Cref{lem:actionLeavingGrowthFactorInvariant} is then derived from \Cref{eq:isotropy,eq:isotropyActionOnHMRepresentation}.
\end{proof}
As it is useless to consider isotropies leaving the growth factor invariant, we limit our search to isotropies of the following form:
\begin{lemma}\label{lem:IwasawaDecomposition}
The action of~\({{({\textup{h}}\times{\textup{p}})}^{\times{3}}}\) determines the growth factor~$\gamma_{2,1}$ for~\({\textup{h}\!=\!\Bigl\lbrace\mat{H}_{\rho}\!=\!\begin{smatrix}\rho&0\\0&\lfrac{1}{\rho}\end{smatrix}\Bigm\vert{{\rho}>{0}}\Bigr\rbrace}\) and~\({\textup{p}\!=\!\Bigl\lbrace\mat{P}_{\xi}=\begin{smatrix}1&\xi\\0&1\end{smatrix}\Bigm\vert{\xi\in\RR}\Bigl\rbrace}\).
\end{lemma}
\begin{proof}
\cref{eq:isotropyActionOnHMRepresentation} shows that the product of any action, say~\(\mat{U}\), by a non-zero scalar affects the growth factor once in~\(\mat{U}\) and once, inverted, in~\(\Inverse{\mat{U}}\), as norms are absolutely homogeneous.
Thus, it is sufficient to consider matrices with determinant~\(1\).
\cref{lem:actionLeavingGrowthFactorInvariant} states that orthogonal matrices do not have any effect.
From the \textsc{qr} decomposition of any invertible matrices, there remains just the~\(({{\textup{h}}\times{\textup{p}}})\) part of~\({\textsc{psl}^{\pm}\bigl({\RR}^{2}\bigr)}\)'s Iwasawa decomposition in~\cref{thm:IsotropiesActTransitivelyOnOptimalAlgorithm}.
\end{proof}
We should study the action of~\({{({\textup{h}}\times{\textup{p}})}^{\times{3}}}\) on Strassen tensor decomposition in order to find variants with the smaller possible~\(\gamma_{2,1}\).
Unfortunately, a direct definitive result for this question seems to be out of reach, and we present several ersatzes.
First, we perform numerical minimization on~\(\GrowthFactor{\IsotropyAction{\Isotropy{g}}{\tensor{S}}}\) with a completely generic isotropy~\(\Isotropy{g}\) in~\({{\textsc{psl}^{\pm}\bigl({\RR}^{2}\bigr)}{}^{\times{3}}}\) (involving~\(6\) indeterminates by~\cref{lem:IwasawaDecomposition}); this experiment suggests that a suitable isotropy to reach a fast matrix product tensor decomposition with minimal~\(\gamma_{2,1}\) could be of the form~\({(\mat{U}\times\mat{U}\times\mat{U})}\) (involving only~\(2\) indeterminates).
The following proposition states precisely this possibility (its proof is a simple second partial derivative test presented in~\cref{sec:ComputationalProofs}).
\begin{restatable}{propositionS}{BestGrowthFactor}\label{prop:BestGrowthFactor}
Consider the matrices~\({\mat{U}(\rho,\xi)=\MatrixProduct{\mat{H}_{\rho}}{\mat{P}_{\xi}}}\) and the isotropies~\(\Isotropy{g}_{\rho,\xi}\) defined by~\({{\mat{U}(\rho,\xi)}{}^{\times{3}}}\).
The minimal value on the orbit~\(\IsotropyAction{\Isotropy{g}_{\rho,\xi}}{\tensor{S}}\) of the growth factor~\(\gamma_{2,1}\bigl({\IsotropyAction{\Isotropy{g}_{\rho,\xi}}{\tensor{S}}}\bigr)\) is reached at the point~\({{(\rho,\xi)}={\bigl(\sqrt[4]{\lfrac{4}{3}},{-\lfrac{1}{2}}\bigr)}}\) and equal to~\({{\lfrac{4}{\sqrt{2}}+\lfrac{16}{\sqrt{3}}}>{12.06603}}\).
\end{restatable}
The algorithm corresponding to the point~\({(\rho,\xi)}\) with minimal~\(\gamma_{2,1}\) on this restricted orbit is given in~\cref{eq:asopt}.
We gather in~\cref{tab:frobenius} values for~\(\gamma_{2,1}\) of some matrix product tensor decompositions, together with the result obtained in~\cref{prop:BestGrowthFactor}.
In~\cref{sec:implem}, we compare the implementation of algorithms associated to these tensor decompositions in order to confirm that their numerical accuracy is correlated to their respective~\(\gamma_{2,1}\) growth factor.
\section{Upper and lower bounds}\label{sec:holder}
We explore in this section some bounds on the norm of each component of an \textsc{hm} representation.
By the multiplicativity of $L_{p,q}$ norms (even generalized to negative
H\"older conjugates), this will always give alternative bounds on the error, a priori less accurate, but potentially easier to apprehend.
\begin{lemma}\label{lem:equiv}
For any \textsc{hm} representation~\(\mathcal{H}\), with matrices~\({\mat{L},\mat{R},\mat{P}}\) in~\({\Field^{r{\times}n}}\), let~\({\gamma_{\mathcal{H}}}\) be its~\(\gamma_{2,1}\) growth factor~\({\gamma_{2,1}(\mathcal{H})}\), as in~\cref{def:growthfactor}.
Then for any strictly positive~\(y\) and~\(z\), we have both:
\begin{gather}\label{eq:normsgammaupper}
\gamma_{\mathcal{H}}\leq\xnorm{\mathcal{H}}{2,3}\leq\Fnorm{\mathcal{H}}\quad\quad\text{and}\\
\max\Bigl\lbrace{r^{1+3z}}\xnorm{\mathcal{H}}{2,{-\frac{1}{z}}};\xnorm{\mat{L}}{2,-\frac{1}{y}}\!{\cdot}\xnorm{\mat{R}}{2,-\frac{1}{z}}\!{\cdot}\xnorm{\mat{P}}{2,\frac{1}{1+y+z}}\Bigr\rbrace
\leq\gamma_{\mathcal{H}.}\label{eq:normsgammalower}
\end{gather}
\end{lemma}
\begin{proof}
Let~\(a_{i}\) (resp.~\({b_{i},c_{i}}\)) denotes~\(\twonorm{\mat{L}_{i}}\) (resp.~\({\twonorm{\mat{R}_{i}},\twonorm{\mat{P}_{i}}}\)).
The first right-hand side inequality is the classical H{\"o}lder's inequality~\({{\onenorm{\vectorif{a_i{\cdot}b_i{\cdot}c_i}{i}} \leq \threenorm{\vectorif{a_i}{i}} {\cdot} \threenorm{\vectorif{b_i}{i}} {\cdot} \threenorm{\vectorif{c_i}{i}}}=\xnorm{\mathcal{H}}{2,3}}\) on~\({a_i, b_i}\) and~\(c_i\) with the H\"older conjugates~\({\frac{1}{3}+\frac{1}{3}+\frac{1}{3}=1}\).
The second right-hand side inequality is a direct application of the monotonicity of norms.
Then, the left-hand side inequality is obtained by a reverse H{\"o}lder's inequality on the vectors~\({a_i, b_i, c_i}\) and~\(1\), with the H\"older conjugates~\({\frac{1}{-1/z}+\frac{1}{-1/z}+\frac{1}{-1/z}+(1+3z)=1}\).
We have indeed that the~\({(1+3z)}\)-norm~\({\xnorm{\vectorif{1}{i}}{1/(1+{3}{z})}}\) is~\({{\bigl(\sum_{i=1}^r{1}^{1/(1+{3}{z})}\bigr)}{}^{1+{3}{z}}}\).
Combined with the relation
\({\xnorm{\mathcal{H}}{2,-\frac{1}{z}}{=}\xnorm{\vectorif{a_i}{i}}{-\frac{1}{z}}{\!\cdot}\xnorm{\vectorif{b_i}{i}}{-\frac{1}{z}}{\!\cdot}\xnorm{\vectorif{c_i}{i}}{-\frac{1}{z}}}\), this shows that the inequality~\({r^{1+3z}\xnorm{\mathcal{H}}{2,-\frac{1}{z}} {\leq} \onenorm{\vectorif{a_i{\cdot}b_i{\cdot}c_i{\cdot}1}{i}}}\) holds.
Finally, for the other lhs, we use H{\"o}lder's inequality on~\({a_i, b_i}\) and~\(c_i\), now with H\"older conjugates~\({\frac{1}{-1/y}+\frac{1}{-1/z}+\frac{1}{1/(1+y+z)}=1}\).
\end{proof}
\begin{table}[ht]%
\centering\setlength{\tabcolsep}{3pt}\renewcommand{\arraystretch}{1.2}
\begin{tabular}{lrcc}
\toprule
Algorithm&\multicolumn{1}{c}{$\gamma_{2,1}({\mathcal{H}})$}&$\xnorm{\mathcal{H}}{2,3}$&$\Fnorm{\mathcal{H}}$\\
\midrule
Winograd &${7{+}\frac{8}{\sqrt{2}}{+}\frac{9}{\sqrt{3}}}\approx{17.853}$&
$11{+}\frac{8}{\sqrt{2}}{+}\frac{9}{\sqrt{3}}$ & $\sqrt{14}^3$\\
Strassen &$12{+}\frac{4}{\sqrt{2}}\approx{14.828}$&
$2{+}\frac{20}{\sqrt{2}}$ & $\sqrt{12}^3$ \\
Eq.~(\ref{eq:powers}) &  $\frac{75}{8}{+}\frac{4}{\sqrt{2}}\approx{12.203}$& $\frac{125}{32}{+}\frac{4}{\sqrt{2}}{+}\frac{25}{2\sqrt{5}}$   & $\sqrt{\frac{162}{16}}^3$\\
Eq.~(\ref{eq:powrot})  & $\frac{75}{8}{+}\frac{4}{\sqrt{2}}\approx{12.203}$& $\frac{125}{32}{+}\frac{4}{\sqrt{2}}{+}\frac{25}{2\sqrt{5}}$& $\sqrt{10}\sqrt{\frac{162}{16}}\frac{810}{80\sqrt{10}}$ \\
Eq.~(\ref{eq:asopt}) & \color{teal}\bf $\frac{16}{\sqrt{3}}{+}\frac{4}{\sqrt{2}}\approx{12.066}$& \color{teal}\bf $\frac{16}{\sqrt{3}}{+}\frac{4}{\sqrt{2}}$ & $\sqrt{10}^3$ \\
Conv.&  $8.000$& $8$ & $\sqrt{8}^3$ \\
\bottomrule
\end{tabular}
\caption{Illustration of~\cref{eq:normsgammaupper} on several~\(\mathcal{H}=\HMRepresentation{\mat{L}}{\mat{R}}{\mat{P}}\)}\label{tab:frobenius}%
\end{table}
\Cref{tab:frobenius} gives the Frobenius and~\({(2,3)}\)-norms of each of the three matrices defining the \textsc{hm} representation of several matrix product algorithms, as well as their~\(\gamma_{2,1}\) growth factor.
\par
In the following proposition, we show that---up to orthogonal
transformations---the minimum of the Frobenius norm of each of the
three \textsc{hm} representation components defining a
fast~\(\matrixsize{2}{2}\)-matrix multiplication algorithms
is~\(\sqrt{10}\).
\begin{restatable}{propositionS}{tencubetwonorm}\label{prop:tencubetwonorm}
The minimal product~\(\Fnorm{\mathcal{H}}\) of the three Frobenius norms of the \textsc{hm} representation of any bilinear algorithm for matrix multiplication with~\(7\) multiplications, is~\(\sqrt{10}{}^3\).
\end{restatable}
This proposition's proof is given in~\cref{sec:ComputationalProofs}.
Remark that this lower bound is reached by the algorithm whose \textsc{hm} representation is given in~\Cref{eq:asopt}.
\begin{equation}\label{eq:asopt}%
\begin{smatrix}
\frac{\sqrt{3}}{2}&\frac{1}{2}&\frac{1}{2}&\frac{\sqrt{3}}{6}\\
0&0&1&{-\frac{\sqrt{3}}{3}}\\
0&1&0&\frac{\sqrt{3}}{3}\\
0&0&0&{-}\frac{2}{\sqrt{3}}\\
{-}\frac{\sqrt{3}}{2}&{-}\frac{1}{2}&\frac{1}{2}&{-}\frac{\sqrt{3}}{2}\\
{-}\frac{\sqrt{3}}{2}&{-}\frac{1}{2}&\frac{1}{2}&\frac{\sqrt{3}}{6}\\
{-}\frac{\sqrt{3}}{2}&\frac{1}{2}&\frac{1}{2}&{-}\frac{\sqrt{3}}{6}\\
\end{smatrix};\
\begin{smatrix}
0&\frac{2}{\sqrt{3}}&0&0\\
-1&\frac{\sqrt{3}}{3}&0&0\\
0&\frac{\sqrt{3}}{3}&0&-1\\
\frac{1}{2}&{-}\frac{\sqrt{3}}{6}&\frac{\sqrt{3}}{2}&{-}\frac{1}{2}\\
{-}\frac{1}{2}&\frac{\sqrt{3}}{2}&{-}\frac{\sqrt{3}}{2}&{-}\frac{1}{2}\\
\frac{1}{2}&\frac{\sqrt{3}}{6}&\frac{\sqrt{3}}{2}&\frac{1}{2}\\
\frac{1}{2}&\frac{\sqrt{3}}{6}&{-}\frac{\sqrt{3}}{2}&{-}\frac{1}{2}\\
\end{smatrix};\
\Transpose{\begin{smatrix}
\frac{\sqrt{3}}{6}&\frac{1}{2}&\frac{1}{2}&\frac{\sqrt{3}}{2}\\
{-}\frac{\sqrt{3}}{3}&0&-1&0\\
\frac{\sqrt{3}}{3}&-1&0&0\\
\frac{\sqrt{3}}{6}&{-}\frac{1}{2}&{-}\frac{1}{2}&\frac{\sqrt{3}}{2}\\
\frac{\sqrt{3}}{2}&{-}\frac{1}{2}&\frac{1}{2}&\frac{\sqrt{3}}{2}\\
{-}\frac{\sqrt{3}}{6}&{-}\frac{1}{2}&\frac{1}{2}&\frac{\sqrt{3}}{2}\\
{-}\frac{2}{\sqrt{3}}&0&0&0\\
\end{smatrix}}.%
\end{equation}
\begin{remark}
Similarly,~\({\bigl(\lfrac{\sqrt[4]{3}}{\sqrt{2}},\lfrac{\sqrt[4]{3}}{\sqrt{6}},\lfrac{\sqrt[4]{3}}{\sqrt{2}},{-}\lfrac{\sqrt[4]{3}}{\sqrt{6}}\bigr)}\)
is a minimum
of~\({\xnormexp{\mat{L}\cdot(\mat{W}\otimes\mat{V})}{2,3}{3}}\) as
in~\cref{prop:tencubetwonorm}
for~\({\xnormexp{\mat{L}\cdot(\mat{W}\otimes\mat{V})}{2,2}{2}}\).
It turns out that this value
is~\({\lfrac{16}{\sqrt{3}}+\lfrac{4}{\sqrt{2}}}\), the same as the
\(\gamma_{2,1}\) growth factor at this point, proving that our upper
bound is reached.
\end{remark}
We now turn to potential lower bounds.
\begin{lemma}\label{lem:minimum}
With~\({\mat{W}=\begin{smatrix}r&x\\0&r^{-1}\end{smatrix}}\),~\({\mat{V}=\begin{smatrix}s&y\\0&s^{-1}\end{smatrix}}\),~\(\mat{L}\)
the first component of Strassen's \textsc{hm} representation given
in~\Cref{eq:StrassenHMRepresentation} and any~\({z\geq{0.5171}}\),
the
point~\({\bigl(\lfrac{\sqrt[4]{3}}{\sqrt{2}},\lfrac{\sqrt[4]{3}}{\sqrt{6}},\lfrac{\sqrt[4]{3}}{\sqrt{2}},{-}\lfrac{\sqrt[4]{3}}{\sqrt{6}}\bigr)}\)
is a local minimum
of~\({\xnorm{\mat{L}\cdot(\mat{W}\otimes\mat{V})}{2,-\lfrac{1}{z}}}\)
as a function of~\({r, x, s}\) and~\(y\).
\end{lemma}
\begin{proof}
As in the proof of~\cref{prop:tencubetwonorm}, we give an explicit expression~\(f_z(r,x,s,y)\) of~\({\xnorm{\mat{L}\cdot(\mat{W}\otimes\mat{V})}{2,-\lfrac{1}{z}}}\)
equal to:
\begin{equation}
{\left(\!\!
\begin{array}{r}
{\bigl((r^2+x^2)(s^2+y^2)+\lfrac{(2xy+\lfrac{1}{rs})}{rs}\bigr)}^{-\lfrac{1}{2z}}+{(rs)}^{\lfrac{1}{z}}\\
+{\bigl(s^2+{(y+\lfrac{1}{s})}^2\bigr)}^{-\lfrac{1}{2z}}\bigl({{(r^2+x^2)}^{-\lfrac{1}{2z}}+{r}^{\lfrac{1}{z}}}\bigr)\\
+{\bigl(r^2+{(x-\lfrac{1}{r})}^2\bigr)}^{-\lfrac{1}{2z}}\bigl({(s^2+y^2)}^{-\lfrac{1}{2z}}+{s}^{\lfrac{1}{z}}\bigr)\\
+{\Bigl(\bigl(r^2+x^2\bigr)\bigl(s^2+y^2\bigr)\Bigr)}^{-\lfrac{1}{2z}}
\end{array}
\right)}^{-z}.
\end{equation}
Then the evaluation of its partial derivatives at the given point is zero, by inspection, for any real~\(z\).
Now, the roots of the characteristic polynomial of the Hessian
of~\(f_z\) at this point
are~\({\frac{n}{6z}\bigl(b_1\pm\sqrt{\delta_1}\bigr)}\)
and~\({\frac{n}{18z}\bigl(b_2\pm\sqrt{\delta_2}\bigr)}\),
with~\({\tau{=}\sqrt[2z]{3}}\) and~\({\lambda{=}\sqrt[2z]{2}}\),
for~\({n{=}{\bigl(\lambda+6\tau\bigr)}^{-1-z}}\), using
\({b_1{=}{{({32z}-{11})}\tau^{1+z}+4z\lambda\sqrt{3}}}\), and
\({b_2{=}{{({{96z}-{63}})\tau^{1+z}}+{52z\lambda\sqrt{3}}}}\),
\(\delta_1{=}24(1-16z)z\tau\lambda+(1344z^2-384z+39)\tau^2+48z^2\lambda^2\)
and finally
\({\delta_2{=}72(272z-237)z\tau\lambda+(12096z^2-20736z+8991)\tau^2+8112z^2\lambda^2}\).

First,~\(\delta_1\) and~\(\delta_2\) are positive for positive $z$, so
that the eigenvalues are then always real.
Second, both expressions~\({{b_{i}}^{\!2}-\delta_{i}}\) have the same
root, strictly less than~\(0.5171\).
Third, all four eigenvalues are thus strictly positive for~\({z\geq{0.5171}}\).
\end{proof}
\begin{corollary}\label{cor:lowerbound}
\({11.7554696<\frac{28}{9}2^{\frac{11}{14}}3^{\frac{5}{7}}}\) is a lower bound for the~\(\gamma_{2,1}\) growth factor of an \textsc{hm} formula using~\(7\) products.
\end{corollary}
\begin{proof}
Following the proof of~\cref{lem:minimum}, we have that equality~\({f_z\!\left(\frac{\sqrt[4]{3}}{\sqrt{2}},\frac{\sqrt[4]{3}}{\sqrt{6}},\frac{\sqrt[4]{3}}{\sqrt{2}},\frac{-\sqrt[4]{3}}{\sqrt{6}}\right)=\!{\left(2^{\frac{-1}{2z}} + 3^{\frac{1}{2z}}2^{\frac{-1}{z}}6\right)}^{-z}}\) holds.
Denoting this quantity by~\(\zeta_{z}\)  we thus have that~\({7^{1+3z}\zeta_z^3}\leq{7^{1+3z}\xnorm{\mathcal{H}}{2,-\lfrac{1}{z}}}\), which left-hand side limit at~\({z=\infty}\) is~\({\frac{28}{9}2^{\frac{11}{14}}3^{\frac{5}{7}}}\).
By~\cref{eq:normsgammalower}, this shows that this value is less than or equal to~\(\gamma_{2,1}\) as announced.
It is also the limit of~\({{{\zeta_{z}}^{\!2}\zeta_{{-1}-2z}}\leq{\xnorm{\mat{L}}{2,-\frac{1}{z}}{\cdot}\xnorm{\mat{R}}{2,-\frac{1}{z}}{\cdot}\xnorm{\mat{P}}{2,\frac{1}{1+z+z}}}}\) at~\({z=\infty}\).
\end{proof}
\Cref{cor:lowerbound} for instance shows that the~\(\gamma_{2,1}\) growth factor of the conventional algorithm ($8$) can not be attained by such fast algorithms.
Let us see now how this bound behaves in our experiments.
\section{Algorithms into practice}\label{sec:implem}
In this section, we present several techniques to lower the number of operations used in our algorithms and thus, lower complexity bounds and potentially obtain a better accuracy.
\par
Determining actual complexity bounds requires estimating the number of operations required to implement a given formula.
Considering an \textsc{hm} representation, a direct upper bound can be obtained by: first count the number of coefficients different from~\({0,\pm{1}}\) to upper bound the number of multiplications/divisions; second count the number of non-zero coefficients, minus the number of rows, to get an upper bound on the number of additions/subtractions.
\par
To obtain lower operation counts, we use the following techniques:
first, we select among equivalently accurate algorithms: this is presented in~\cref{ssec:rotations};
second, we factor as much as possible the computations between rows of the \textsc{hm} representations, as in~\cref{ssec:factor};
third, we use dependent rows as more opportunities for factorization, as in~\cref{ssec:kernel}.
We then present some good candidates (as well as in~\cref{app:suppl})
and we eventually look at some potential sparse alternative change of basis in~\cref{ssec:schwartz}.
\subsection{Sparsifying via rotations}\label{ssec:rotations}
We have seen in~\cref{lem:actionLeavingGrowthFactorInvariant} that orthogonal transformations leave the Frobenius norm invariant and thus, the~\(\gamma_{2,1}\) growth factor.
Therefore, one can apply~\(\matrixsize{4}{4}\) generic Kronecker products of orthogonal~\(\matrixsize{2}{2}\) (rotation) matrices using~\cref{lem:actionOnHMRepresentation} and try to optimize the considered \textsc{hm} representation for several possible goals:
(1) %
a smaller number of non-zero coefficients in \textsc{hm} representation components;
(2) %
a non-zero pattern better suited to factorization (see the technique of~\cref{ssec:factor});
(3) %
a triangular (sparse) subset of independent rows (see the technique of~\cref{ssec:kernel}).
\par
For instance, to obtain~\cref{eq:asopt}, we solve for the minimal values of the Frobenius norms as in~\cref{prop:tencubetwonorm} and then for orthogonal transformations that produce as many vectors of the canonical basis as possible.
Doing so, we found that with~\(\gamma_{2,1}\) set to~\({\lfrac{16}{\sqrt{3}}+\lfrac{4}{\sqrt{2}}}\) and \textsc{hm} representation component Frobenius norms set to~\(\sqrt{10}\), the maximal possible number of canonical vectors was~$1$.
\Cref{eq:asopt} is one of those.
Similarly, \Cref{eq:powrot} is an orthogonal optimization of~\Cref{eq:powers}, with one canonical vector in each of components of the \textsc{hm} representation.
A \textsc{c}++ implementation of these tools is available in the \plinopt~library~\cite{jgd:2024:plinopt}.
\subsection{Factoring heuristics}\label{ssec:factor}
For the implementation of a given linear operator (in this work one of the matrices in the \textsc{hm} representation) one can try to find the shortest straight-line program for its computation.
The problem is \textsc{np}-hard in general (see e.g.~\cite[\S~3.1.1]{Boyar:2013aa}); but for small matrices, and over the field with~\(2\) elements,~\cite{Boyar:2013aa} and references therein propose several heuristics that potentially reduce the number of operations.
\par
Not all of them are applicable to fields with more elements, but we use a kind of common sub-expression eliminations, the ``cancellation-free'' search, described in~\cref{alg:factoringout} and implemented in \href{https://github.com/jgdumas/plinopt/blob/main/src/optimizer.cpp}{\texttt{plinopt/optimizer\,-D}}~\cite{jgd:2024:plinopt}.
\subsection{Kernel computation and Tellegen's principle}\label{ssec:kernel}
If the rank of the linear operator is lower than its number of rows, then an additional strategy has proven useful:
compute first some independent rows, then express the dependent ones by their linear relations.
For this, \cref{alg:kernel} computes a left kernel of the linear operator and uses it to compute the dependent rows via linear combinations of the independent ones.
This is sometimes faster than directly computing the dependent rows.
Of course, if the matrix's rank is lower than the number of \emph{columns}, one can apply~\cref{alg:kernel} to the transposed matrix and then apply the \emph{Tellegen's transposition} principle to recover the transposed linear dependencies (e.g.\ see~\cite{bostan:2003} and references therein).
\begin{algorithm}[h]
\caption{Kernel decomposition of a linear operator}\label{alg:kernel}
\begin{algorithmic}[1]
\Require{\(\mat{M}\) in~\(\Field^{m{\times}n}\) such that~\({r=\MatrixRank{\mat{M}}}\).}
\Ensure{A straight line program computing~\(\vec{u}\leftarrow{\mat{M}}{\cdot}\vec{v}\).}
\State{By Gaussian elimination, compute~\({\mat{M}={\mat{P}\cdot\mat{L}\cdot\mat{U}\cdot\mat{Q}}}\) with~\(\mat{P}\) a permutation matrix,~\(\mat{L}\) in~\(\Field^{m{\times}r}\) be~\({\begin{smatrix}\mat{L}_1\\ \mat{L}_2\end{smatrix}}\) unit upper triangular and~\(\mat{L}_1\) in~\(\Field^{r{\times}r}\); choosing~\(\mat{P}\) so that~(1) the first~\(r\) rows of~\(\mat{P}^{-1}\mat{M}\) are sparsest;~(2)~\(\mat{L}_1\) is the sparsest;~(3)~\(\mat{L}_2\) is the sparsest;}
\State{Let~\(\sigma\) be the permutation represented by~\(\mat{P}\);}
\State{Apply Alg.~\ref{alg:factoringout} to~\({\Transpose{\begin{smatrix}u_{\sigma(1)}\ldots{u_{\sigma(r)}}\end{smatrix}}\leftarrow\begin{smatrix}I_r 0\end{smatrix}\cdot\mat{P}\cdot\mat{M}\cdot\vec{v}}\);}
\Statex\Comment{\(\begin{smatrix}-{\mat{L}_{2}}\cdot{{\mat{L}_{1}}^{\!{-1}}}&\mat{I}_{m-r}\end{smatrix}\) is a (sparse) left kernel of~\(\mat{M}\) and provides the linear dependencies of the remaining rows}
\State{Apply Alg.~\ref{alg:factoringout} to~\({\Transpose{\begin{smatrix}u_{\sigma(r+1)}\ldots{u_{\sigma(m)}}\end{smatrix}}\leftarrow{\mat{L}_{2}}\cdot{{\mat{L}_{1}}^{\!{-1}}}\Transpose{\begin{smatrix}u_{\sigma(1)}\ldots{u_{\sigma(r)}}\end{smatrix}}}\).}
\end{algorithmic}
\end{algorithm}
\par
\Cref{alg:kernel} is implemented in \href{https://github.com/jgdumas/plinopt/blob/main/src/optimizer.cpp}{\texttt{plinopt/optimizer\,-K}}.
The Tellegen's transposition principle applied to such \textsc{slp}s is implemented in \href{https://github.com/jgdumas/plinopt/blob/main/src/transpozer.cpp}{\texttt{plinopt/transpozer}}~\cite{jgd:2024:plinopt,jgd:2024:mFMM}.
These routines have produced the implementations for our different \textsc{hm} formulas given in the following section (e.g.\ the implementation~\cref{alg:asopt} of~\cref{eq:asopt} with only~\(24\) additions and~\(12\) multiplications/divisions).
\begin{table}[h]%
\centering\fbox{\begin{minipage}{.95\linewidth}%
\[\setlength\arraycolsep{2pt}\renewcommand{\arraystretch}{1}
\begin{array}{llll}
t_1=\frac{\sqrt{3}}{3}a_{22}& t_2=a_{21}+t_1&s_1=\frac{\sqrt{3}}{3}b_{21}& s_2=s_1-b_{11}\\
t_3=a_{12}+t_2& l_1=\frac{\sqrt{3}}{2}a_{11}+\frac{1}{2}t_3&s_3=s_2+b_{22}& r_1=2s_1\\
l_2=a_{12}-t_1& l_3=t_2&r_2=s_2& r_3=s_1-b_{22}\\
l_4=2t_1& l_5=l_2-l_1&r_4=\frac{1}{2}s_3{-}\frac{\sqrt{3}}{2}b_{12}& r_5=r_3+r_4\\
l_6=l_5+l_4& l_7=l_5+l_3&r_6=r_1-r_5& r_7=r_5-r_2
\end{array}
\]
\[%
\begin{array}{llll}
p_1=l_1{\cdot}r_1&
p_2=l_2{\cdot}r_2&
p_3=l_3{\cdot}r_3&
p_4=l_4{\cdot}r_4\\
p_5=l_5{\cdot}r_5&
p_6=l_6{\cdot}r_6&
p_7=l_7{\cdot}r_7
\end{array}
\]
\[\setlength\arraycolsep{2pt}\renewcommand{\arraystretch}{1}
\begin{array}{llll}
w_2=p_5+p_1+p_6&
w_1=p_7+p_6&
w_3=w_2-p_2&
w_5=\frac{p_4+w_2}{2}\\
c_{12}=p_1-p_3-w_5&
c_{21}=w_3-w_5&
c_{22}=\sqrt{3}w_5\\
\multicolumn{4}{c}{c_{11}=\frac{\sqrt{3}}{3}(w_3-c_{12}-2w_1)}
\end{array}
\]
\end{minipage}}
\caption{\textsc{slp} of~\cref{eq:asopt} with~\(24\) add.\ and~\(12\) mul./div.}\label{alg:asopt}
\end{table}
\begin{figure}[htb]
\caption{Numerical accuracy vs size (normal distribution)}\label{fig:gamma}
\Description[Numerical accuracy vs size (normal
distribution)]{Numerical accuracy vs size (normal distribution)}
\begingroup
  \makeatletter
  \providecommand\color[2][]{%
    \GenericError{(gnuplot) \space\space\space\@spaces}{%
      Package color not loaded in conjunction with
      terminal option `colourtext'%
    }{See the gnuplot documentation for explanation.%
    }{Either use 'blacktext' in gnuplot or load the package
      color.sty in LaTeX.}%
    \renewcommand\color[2][]{}%
  }%
  \providecommand\includegraphics[2][]{%
    \GenericError{(gnuplot) \space\space\space\@spaces}{%
      Package graphicx or graphics not loaded%
    }{See the gnuplot documentation for explanation.%
    }{The gnuplot epslatex terminal needs graphicx.sty or graphics.sty.}%
    \renewcommand\includegraphics[2][]{}%
  }%
  \providecommand\rotatebox[2]{#2}%
  \@ifundefined{ifGPcolor}{%
    \newif\ifGPcolor
    \GPcolortrue
  }{}%
  \@ifundefined{ifGPblacktext}{%
    \newif\ifGPblacktext
    \GPblacktexttrue
  }{}%
  \let\gplgaddtomacro\g@addto@macro
  \gdef\gplbacktext{}%
  \gdef\gplfronttext{}%
  \makeatother
  \ifGPblacktext
    \def\colorrgb#1{}%
    \def\colorgray#1{}%
  \else
    \ifGPcolor
      \def\colorrgb#1{\color[rgb]{#1}}%
      \def\colorgray#1{\color[gray]{#1}}%
      \expandafter\def\csname LTw\endcsname{\color{white}}%
      \expandafter\def\csname LTb\endcsname{\color{black}}%
      \expandafter\def\csname LTa\endcsname{\color{black}}%
      \expandafter\def\csname LT0\endcsname{\color[rgb]{1,0,0}}%
      \expandafter\def\csname LT1\endcsname{\color[rgb]{0,1,0}}%
      \expandafter\def\csname LT2\endcsname{\color[rgb]{0,0,1}}%
      \expandafter\def\csname LT3\endcsname{\color[rgb]{1,0,1}}%
      \expandafter\def\csname LT4\endcsname{\color[rgb]{0,1,1}}%
      \expandafter\def\csname LT5\endcsname{\color[rgb]{1,1,0}}%
      \expandafter\def\csname LT6\endcsname{\color[rgb]{0,0,0}}%
      \expandafter\def\csname LT7\endcsname{\color[rgb]{1,0.3,0}}%
      \expandafter\def\csname LT8\endcsname{\color[rgb]{0.5,0.5,0.5}}%
    \else
      \def\colorrgb#1{\color{black}}%
      \def\colorgray#1{\color[gray]{#1}}%
      \expandafter\def\csname LTw\endcsname{\color{white}}%
      \expandafter\def\csname LTb\endcsname{\color{black}}%
      \expandafter\def\csname LTa\endcsname{\color{black}}%
      \expandafter\def\csname LT0\endcsname{\color{black}}%
      \expandafter\def\csname LT1\endcsname{\color{black}}%
      \expandafter\def\csname LT2\endcsname{\color{black}}%
      \expandafter\def\csname LT3\endcsname{\color{black}}%
      \expandafter\def\csname LT4\endcsname{\color{black}}%
      \expandafter\def\csname LT5\endcsname{\color{black}}%
      \expandafter\def\csname LT6\endcsname{\color{black}}%
      \expandafter\def\csname LT7\endcsname{\color{black}}%
      \expandafter\def\csname LT8\endcsname{\color{black}}%
    \fi
  \fi
    \setlength{\unitlength}{0.0500bp}%
    \ifx\gptboxheight\undefined%
      \newlength{\gptboxheight}%
      \newlength{\gptboxwidth}%
      \newsavebox{\gptboxtext}%
    \fi%
    \setlength{\fboxrule}{0.5pt}%
    \setlength{\fboxsep}{1pt}%
    \definecolor{tbcol}{rgb}{1,1,1}%
\begin{picture}(4320.00,3960.00)%
    \gplgaddtomacro\gplbacktext{%
      \csname LTb\endcsname
      \put(219,538){\makebox(0,0)[r]{\strut{}$10^{-14}$}}%
      \csname LTb\endcsname
      \put(219,1392){\makebox(0,0)[r]{\strut{}$10^{-13}$}}%
      \csname LTb\endcsname
      \put(219,2246){\makebox(0,0)[r]{\strut{}$10^{-12}$}}%
      \csname LTb\endcsname
      \put(219,3100){\makebox(0,0)[r]{\strut{}$10^{-11}$}}%
      \csname LTb\endcsname
      \put(338,174){\makebox(0,0){\strut{}$32$}}%
      \csname LTb\endcsname
      \put(1327,174){\makebox(0,0){\strut{}$64$}}%
      \csname LTb\endcsname
      \put(2317,174){\makebox(0,0){\strut{}$128$}}%
      \csname LTb\endcsname
      \put(3307,174){\makebox(0,0){\strut{}$256$}}%
      \csname LTb\endcsname
      \put(4297,174){\makebox(0,0){\strut{}$512$}}%
    }%
    \gplgaddtomacro\gplfronttext{%
      \csname LTb\endcsname
      \put(728,3608){\makebox(0,0)[l]{\strut{}Winograd \cite{Winograd:1977:complexite}}}%
      \csname LTb\endcsname
      \put(728,3434){\makebox(0,0)[l]{\strut{}Strassen \cite{strassen:1969}}}%
      \csname LTb\endcsname
      \put(728,3260){\makebox(0,0)[l]{\strut{}\cref{alg:powers,eq:powers}}}%
      \csname LTb\endcsname
      \put(728,3085){\makebox(0,0)[l]{\strut{}\cref{alg:powrot,eq:powrot}}}%
      \csname LTb\endcsname
      \put(728,2911){\makebox(0,0)[l]{\strut{}\cref{alg:0695}}}%
      \csname LTb\endcsname
      \put(728,2737){\makebox(0,0)[l]{\strut{}\cref{alg:0661}}}%
      \csname LTb\endcsname
      \put(728,2562){\makebox(0,0)[l]{\strut{}\cref{alg:asopt,eq:asopt}}}%
      \csname LTb\endcsname
      \put(728,2388){\makebox(0,0)[l]{\strut{}Conventional}}%
      \csname LTb\endcsname
      \put(0,3887){\makebox(0,0){\strut{}error}}%
      \csname LTb\endcsname
      \put(2296,0){\makebox(0,0){\strut{}square matrix dimension}}%
    }%
    \gplbacktext
    \put(0,0){\includegraphics[width={216.00bp},height={198.00bp}]{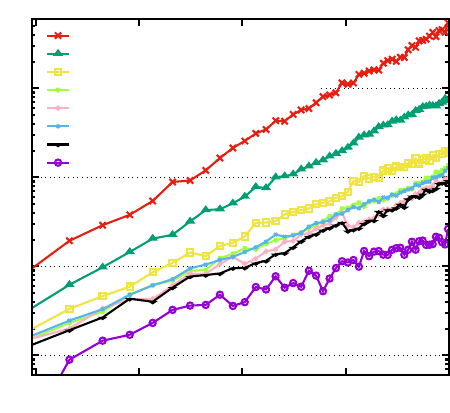}}%
    \gplfronttext
  \end{picture}%
\endgroup
\end{figure}

\begin{remark}
The accuracy obtained with our different fast variants is given in~\cref{fig:gamma} using the Matlab framework of~\cite{Dai:2023aa}, which we forked in~\cite{jgd:2024:mFMM} and where we have just added the implementations of the variants presented here.
Thus, in~\cref{fig:gamma,fig:Schwartz,fig:badcond} we present the error as the infinity norm of the difference between the result of our implementations and the exact matrix multiplication.
\end{remark}
In~\cref{fig:gamma}, all our variants, \cref{alg:powers,alg:powrot,alg:0695,alg:0661,alg:asopt} with decreasing~\(\gamma_{2,1}\), are mostly more and more accurate.
Our best algorithm presents an order of magnitude advantage over Strassen's and two orders of magnitude advantage over Winograd's.
It is then quite close to the conventional algorithm's accuracy.
\Cref{fig:gamma} uses normal distribution, the same behavior is obtained, e.g., with a uniform distribution (see corresponding code in~\cite{jgd:2024:mFMM}).
\begin{remark}
In~\cite{bini:1980} the authors consider all bilinear algorithms using~\(7\) multiplications with constants of the form~\(\pm{2^i}\); they showed that Strassen's original method~\cite{strassen:1969} reaches in this class the minimum value~\(12\) of their~\(\gamma_{0,1\infty}\) factor error bound (while for instance that of Winograd~\cite{Winograd:1977:complexite} is~\(18\), see also~\cref{tab:bounds:sota}).
\par
We propose an algorithm in this class that has a worse~\(\gamma_{0,1\infty}\) of~\(40\), but a~\(\gamma_{2,1}\) of~\({\lfrac{4}{\sqrt{2}}+\lfrac{75}{8}\approx{12.2034}}\), better than those of Strassen~\(14.8284\) or Winograd~\(17.8530\) (see~\cref{tab:frobenius}). This algorithm is defined by the following~\textsc{hm} representation:
\begin{equation}\label{eq:powers}
\begin{smatrix}
0&-1&1&0\\
1&\frac{1}{2}&{-}\frac{1}{2}&{-}\frac{1}{4}\\
0&0&1&{-}\frac{1}{2}\\
0&1&0&{-}\frac{1}{2}\\
0&0&1&\frac{1}{2}\\
1&{-}\frac{1}{2}&\frac{1}{2}&{-}\frac{1}{4}\\
0&1&0&\frac{1}{2}\\
\end{smatrix};\quad
\begin{smatrix}
1&0&0&-1\\
1&\frac{1}{2}&0&0\\
0&\frac{1}{2}&0&-1\\
\frac{1}{2}&\frac{1}{4}&-1&{-}\frac{1}{2}\\
0&\frac{1}{2}&0&1\\
1&{-}\frac{1}{2}&0&0\\
\frac{1}{2}&{-}\frac{1}{4}&1&{-}\frac{1}{2}\\
\end{smatrix};\quad
\Transpose{\begin{smatrix}
0 & 1 & 1 & 0 \\
\frac{1}{2} & 1 & 0 & 0 \\
\frac{1}{4} & {-}\frac{1}{2} & {-}\frac{1}{2} & 1 \\
{-}\frac{1}{2} & 0 & 1 & 0 \\
\frac{1}{4} & \frac{1}{2} & \frac{1}{2} & 1 \\
\frac{1}{2} & -1 & 0 & 0 \\
\frac{1}{2} & 0 & 1 & 0 \\
\end{smatrix}}.
\end{equation}
\Cref{fig:gamma} shows that the induced algorithm given in~\Cref{alg:powers} is also more accurate in practice than both Strassen's and Winograd's variants.
\begin{table}[ht]%
\fbox{%
\begin{minipage}{.95\linewidth}%
\[\setlength\arraycolsep{4pt}\renewcommand{\arraystretch}{1}
\begin{array}{llll}
r_{1} = \frac{1}{2}a_{22}& t_{2} = a_{21}-r_{1} &
u_{1} = \frac{1}{2}b_{12}& s_{1} = b_{11}+u_{1}\\
t_{3} = a_{12}+r_{1}&t_{0}=t_{2}-t_{3}&
s_{2} = u_{1}-b_{22}& u_{2} = s_{1}-b_{22}\\
t_{4} = a_{21}+r_{1}& r_{2}=t_{2}-a_{12}&
s_{4} = b_{22}+u_{1}&s_{0}=s_{1}-s_{4} \\
t_{5} = a_{11}+\frac{1}{2}r_{2} & t_{1}=t_{5}-t_{0}&
s_{3} = \frac{1}{2}u_{2}-b_{21} & s_{5}=s_{0}-s_{2}
\end{array}
\]
\[%
\begin{array}{llll}
p_{1}=t_{0}{\cdot}s_{0}&
p_{2}=t_{1}{\cdot}s_{1}&
p_{3}=t_{2}{\cdot}s_{2}&
p_{4}=t_{3}{\cdot}s_{3}\\
p_{5}=t_{4}{\cdot}s_{4}&
p_{6}=t_{5}{\cdot}s_{5}&
\multicolumn{2}{l}{p_{7}=(t_{4}{-}t_{0}){\cdot}(s_{0}{-}s_{3})}
\end{array}
\]
\[\setlength\arraycolsep{4pt}\renewcommand{\arraystretch}{1}
\begin{array}{llll}
c_{22}{=}p_{5}{+}p_{3}&
v_{1}{=}p_{1}{-}p_{6}{-}p_{3}&
v_{2}{=}p_{7}{+}p_{6}&
v_{3}{=}p_{4}{+}v_{1}\\
v_{4}{=}\frac{c_{22}}{2}&
c_{12}{=}p_{2}{+}v_{1}{+}v_{4}&
c_{21}{=}v_{2}{+}v_{3}{+}v_{4}&
c_{11}{=}\frac{(c_{12}+v_{2}-v_{3})}{2}
\end{array}
\]
\end{minipage}
}
\caption{\textsc{slp} of~\cref{eq:powers} with~\(27\) add.,~\(6\) div.\ by~\(2\) and~\({\gamma_{F}\!\approx\!{12.2034}}\)}\label{alg:powers}
\end{table}
\end{remark}
\begin{remark}\label{rem:rationalHigherOrderApproximations}
\cref{eq:powers} was obtained by approximating the minimal point of the~\(\gamma_{2,1}\) growth factor taken from \Cref{prop:BestGrowthFactor} with the smallest powers of~\(2\).
Further rational higher-order approximations are obtained in the same vein, giving for instance~\cref{eq:powrot,alg:powrot,alg:0695,alg:0661}, as shown in~\cref{sec:rationalapproximations}.
\end{remark}
\subsection{Alternative basis sparsification}\label{ssec:schwartz}
The technique of~\cite{Karstadt:2017aa,Beniamini:2019aa} reduces the number of operations by factoring each matrix in the \textsc{hm} decomposition into a sparser one via a~\(\matrixsize{4}{4}\) change of basis (\textsc{CoB}).
In a recursive version, the left and right-hand sides (resp.\ result) of considered \textsc{CoB} can be recursively precomputed (resp.\ post-computed),  for a total cost in~\({\bbigO{n^{2}\log{n}}}\).
In the meantime the sparser~\(\matrixsize{7}{4}\) matrices are applied, reducing the dominant term of the computation.
The optimal decomposition of Winograd's algorithm in~\cite[\S~3.3]{Karstadt:2017aa} reduces the number of intermediate additions from~\(15\) to~\(12\).
For a fully recursive version, this reduces the leading term in the cost bound from~\(6{n}^{\log_{2}{\!7}}\) to~\(5{n}^{\log_{2}{\!7}}\).
\par
Applying  this approach to the algorithm of~\cref{eq:asopt}, using  the \textsc{CoB} of~\cref{eq:alternative} leads to the sparser \textsc{HM} representation in~\cref{eq:schwartz}.
The leading term of the cost bound thus is reduced from~\(13n^{\log_{2}{\!7}}\) for~\cref{alg:asopt} to~\(5n^{\log_{2}{\!7}}\) for~\cref{alg:LCoB,alg:RCoB,alg:schwartzopt,alg:CoBP,alg:CoBschwartzopt}.
\begin{align}\label{eq:alternative}
\begin{smatrix}
0&0&0&\frac{2}{\sqrt{3}}\\
0&1&0&\frac{\sqrt{3}}{3}\\
0&0&1&{-}\frac{\sqrt{3}}{3}\\
{-}\frac{\sqrt{3}}{2}&{-}\frac{1}{2}&\frac{1}{2}&{-}\frac{\sqrt{3}}{2}\\
\end{smatrix}&;
&\begin{smatrix}
0&\frac{2}{\sqrt{3}}&0&0\\
1&{-}\frac{\sqrt{3}}{3}&0&0\\
0&\frac{\sqrt{3}}{3}&0&-1\\
{-}\frac{1}{2}&\frac{\sqrt{3}}{2}&{-}\frac{\sqrt{3}}{2}&{-}\frac{1}{2}\\
\end{smatrix}&;&
\Transpose{\begin{smatrix}
{-}\frac{2}{\sqrt{3}}&0&0&0\\
\frac{\sqrt{3}}{3}&-1&0&0\\
{-}\frac{\sqrt{3}}{3}&0&-1&0\\
\frac{\sqrt{3}}{2}&{-}\frac{1}{2}&\frac{1}{2}&\frac{\sqrt{3}}{2}\\
\end{smatrix}}\!\!,&
\\\label{eq:schwartz}
\begin{smatrix}
0&0&1&-1\\
0&0&1&0\\
0&1&0&0\\
-1&0&0&0\\
0&0&0&1\\
1&0&0&1\\
0&1&0&1\\
\end{smatrix}&;
&\begin{smatrix}
1&0&0&0\\
0&-1&0&0\\
0&0&1&0\\
0&0&1&-1\\
0&0&0&1\\
1&0&0&-1\\
0&1&0&1\\
\end{smatrix}&;&
\Transpose{\begin{smatrix}
0&-1&0&1\\
0&0&1&0\\
0&1&0&0\\
0&0&1&1\\
0&0&0&1\\
1&0&0&1\\
1&0&0&0\\
\end{smatrix}}\!\!.&
\end{align}
To obtain this \textsc{CoB}, the generic technique of~\cite{Beniamini:2019aa} can be used.
In our case, for~\(\matrixsize{4}{4}\) \textsc{CoB}, the following heuristic was sufficient to obtain optimal (\(12\)-additions) sparse~\({(0,\pm{1})}\) intermediate matrices:
(1) Find independent columns of each \textsc{CoB} one at a time;
(2) For this, alternatively factor-out common coefficients in the resulting columns and find a linear combination minimizing the density of the resulting column, using as coefficients of the combination only in~\({\lbrace{-1},0,1\rbrace}\) and some of the values of the coefficients of the input;
(3) Until this alternation does not sparsify anymore.
This heuristic is implemented in \href{https://github.com/jgdumas/plinopt/blob/main/src/sparsifier.cpp}{\texttt{plinopt/sparsifier}}~\cite{jgd:2024:plinopt} and the resulting  implementation is shown in~\cref{alg:LCoB,alg:RCoB,alg:schwartzopt,alg:CoBP,alg:CoBschwartzopt}.
\newcommand{\LCoB}{\ensuremath{\text{LCoB}}}
\begin{algorithm}[htbp]
\caption{\LCoB$(\mat{A},\ell)$ left change-of-basis of~\cref{eq:alternative}}\label{alg:LCoB}
\begin{algorithmic}[1]
\IfThenEnd{$\ell\leq{0}$}{\Return{$\mat{A}$.}}
\State{%
$m_1=\LCoB(a_{11},\ell{-}1)$;
$m_2=\LCoB(a_{21},\ell{-}1)$;}
\State{%
$m_3=\LCoB(a_{12},\ell{-}1)$;
$m_4=\LCoB(a_{22},\ell{-}1)$;}
\State{%
$t_1 = \frac{1}{\sqrt{3}}m_4$;
$t_2 = m_3-m_2$;
$t_3 = m_1+m_4$;}
\State{\Return{$\left[\frac{2}{\sqrt{3}}m_4,
m_2+t_1,
m_3-t_1,
\frac{1}{2}t_2-\frac{\sqrt{3}}{2}t_3\right]$.}}
\end{algorithmic}
\end{algorithm}
\newcommand{\RCoB}{\ensuremath{\text{RCoB}}}
\begin{algorithm}[htbp]
\caption{\RCoB$(\mat{A},\ell)$ right change-of-basis of~\cref{eq:alternative}}\label{alg:RCoB}
\begin{algorithmic}[1]
\IfThenEnd{$\ell\leq{0}$}{\Return{$\mat{A}$.}}
\State{%
$m_1=\RCoB(a_{11},\ell{-}1)$;
$m_2=\RCoB(a_{21},\ell{-}1)$;}
\State{%
$m_3=\RCoB(a_{12},\ell{-}1)$;
$m_4=\RCoB(a_{22},\ell{-}1)$;}
\State{%
$t_1 =\frac{1}{\sqrt{3}}m_2$;
$t_2 = m_1+m_4$;
$t_3 = m_2-m_3$;}
\State{\Return{$\left[\frac{2}{\sqrt{3}}m_2,
m_1-t_1,
t_1-m_4,
\frac{\sqrt{3}}{2}t_3-\frac{1}{2}t_{2}\right]$.}}
\end{algorithmic}
\end{algorithm}
\newcommand{\CoBP}{\ensuremath{\text{CoBP}}}
\begin{algorithm}[htbp]
\caption{\CoBP\({(\mat{A},\ell)}\) product change-of-basis of~\cref{eq:alternative}}\label{alg:CoBP}
\begin{algorithmic}[1]
\IfThenEnd{$\ell\leq{0}$}{\Return{\(\mat{A}\)}}
\State{%
$m_1=\CoBP(a_{11},\ell{-}1)$;
$m_2=\CoBP(a_{21},\ell{-}1)$;}
\State{%
$m_3=\CoBP(a_{12},\ell{-}1)$;
$m_4=\CoBP(a_{22},\ell{-}1)$;}
\State{%
$t_1 = \frac{1}{2}m_4$;
$t_2 = m_2-m_3$;
$t_3 = \frac{\sqrt{3}}{2}m_4$;}
\State{\Return{\(\left[t_3+\frac{1}{\sqrt{3}}t_2-m_1\frac{2}{\sqrt{3}},-m_2-t_1,t_1-m_3,t_3\right]\).}}
\end{algorithmic}
\end{algorithm}
\begin{algorithm}[htbp]
\caption{Sparsification applied to~\cref{eq:asopt} (via~\cref{eq:alternative,eq:schwartz})}\label{alg:CoBschwartzopt}
\begin{algorithmic}[1]
\Require{$\mat{A},\mat{B}\in\Field^{{{n_{0}}2^\ell}\times{{n_{0}}2^\ell}}$.}
\Ensure{$\mat{C}=\MatrixProduct{A}{B}$.}
\State{$\bar{\mat{A}}\leftarrow\LCoB(\mat{A},\ell)$; $\bar{\mat{B}}\leftarrow\RCoB(\mat{B},\ell)$ \hfill\Comment{Via~\cref{alg:LCoB,alg:RCoB}}}
\State{$\bar{\mat{C}}\leftarrow\bar{\mat{A}}\cdot\bar{\mat{B}}$; \hfill\Comment{Via~\cref{alg:schwartzopt} with~$\ell$ recursive calls}}
\State{\Return{$\mat{C}\leftarrow\CoBP(\bar{\mat{C}},\ell)$.\hfill\Comment{Via~\cref{alg:CoBP}}}}
\end{algorithmic}
\end{algorithm}
\begin{table}[ht]%
\fbox{%
\begin{minipage}{.95\linewidth}\vspace{-5pt}
\[%
\begin{array}{lll}
s_1=a_{11}+a_{12}&
s_2=a_{11}+a_{22}&
s_3=a_{11}-a_{21}\\
t_1=b_{12}+b_{22}&
t_2=b_{11}+b_{12}&
t_3=b_{12}+b_{21}
\end{array}
\]
\vspace{-5pt}
\[%
\begin{array}{llll}
p_1 = a_{11}{\cdot}b_{12} &
p_2 = s_1{\cdot}b_{21} &
p_3 = a_{21}{\cdot}t_1 &
p_4 = a_{12}{\cdot}t_2 \\
p_5 = s_2{\cdot}b_{22} &
p_6 = a_{22}{\cdot}t_3 &
p_7 = s_3{\cdot}b_{11}
\end{array}\]
\vspace{-5pt}
\[%
\begin{array}{llll}
c_{11}{=}p_7{-}p_6 &
c_{12}{=}p_2{+}p_3 &
c_{21}{=}p_4{-}p_5 &
c_{22}{=}p_1{+}p_2{+}p_5{+}p_6
\end{array}\]
\vspace{-10pt}
\end{minipage}}
\caption{\textsc{slp} of~\cref{eq:schwartz} with~\(12\) additions}\label{alg:schwartzopt}
\end{table}

The sparsification process improves the  \(\gamma_{2,1}\) growth factor : applied to
Winograd's original algorithm~\cite{Winograd:1977:complexite}, it goes down from
\({{7}+\lfrac{8}{\sqrt{2}}+\lfrac{9}{\sqrt{3}} \approx{17.853}}\)
to \({4+\lfrac{12}{\sqrt{2}}\approx{12.486}}\) in~\cite{Karstadt:2017aa} and applied
on~\cref{eq:asopt}, from \({{{\lfrac{4}{\sqrt{2}}}+{\lfrac{16}{\sqrt{3}}}}\approx{\!12.066}}\) to~\({{7+\lfrac{6}{\sqrt{2}}}\approx{11.243}}\).

However the resulting sparser bilinear operator itself no longer correspond to a matrix multiplication
algorithm, and this has two consequences: first, the error of this operator only follows the weaker
bound~\Cref{eq:opaccuracy} of~\Cref{th:recbound} (with an additionnal logarithmic factor). Second,
the error bound for the resulting matrix multiplication algorithm must then also include
the contribution of the \textsc{CoB}. A tight analysis of this contribution is made in the recent
work of~\cite[Th.~I.1]{Schwartz:2024aa}. There, the first error bound for an alternative
basis based matrix multiplication algorithm is produced, in the form of \Cref{eq:opaccuracy}
where the contribution of the \textsc{CoB} only affects the multiplicative constant $Q_0$ by a small amount.
\begin{remark}
\cref{alg:CoBschwartzopt}, enjoys simultaneously
the best known leading term in the cost bound, and a close to the best known numerical
accuracy for sub-cubic $2\times2$ algorithms.
The former property comes from the fact that \cref{eq:schwartz} requires only~\(12\) additions.
The latter is shown in practice in~\cref{fig:Schwartz}, and  the proof of~\cite[Th.~I.1]{Schwartz:2024aa}
could be adapted to this algorithm for a theoretical error bound.
\end{remark}
\begin{figure}[htbp]
\caption{Numerical effect of sparsification (\footnotesize{normal distribution})}\label{fig:Schwartz}
\Description[Numerical effect of sparsification]{Numerical effect of sparsification}
%
\begingroup
  \makeatletter
  \providecommand\color[2][]{%
    \GenericError{(gnuplot) \space\space\space\@spaces}{%
      Package color not loaded in conjunction with
      terminal option `colourtext'%
    }{See the gnuplot documentation for explanation.%
    }{Either use 'blacktext' in gnuplot or load the package
      color.sty in LaTeX.}%
    \renewcommand\color[2][]{}%
  }%
  \providecommand\includegraphics[2][]{%
    \GenericError{(gnuplot) \space\space\space\@spaces}{%
      Package graphicx or graphics not loaded%
    }{See the gnuplot documentation for explanation.%
    }{The gnuplot epslatex terminal needs graphicx.sty or graphics.sty.}%
    \renewcommand\includegraphics[2][]{}%
  }%
  \providecommand\rotatebox[2]{#2}%
  \@ifundefined{ifGPcolor}{%
    \newif\ifGPcolor
    \GPcolortrue
  }{}%
  \@ifundefined{ifGPblacktext}{%
    \newif\ifGPblacktext
    \GPblacktexttrue
  }{}%
  \let\gplgaddtomacro\g@addto@macro
  \gdef\gplbacktext{}%
  \gdef\gplfronttext{}%
  \makeatother
  \ifGPblacktext
    \def\colorrgb#1{}%
    \def\colorgray#1{}%
  \else
    \ifGPcolor
      \def\colorrgb#1{\color[rgb]{#1}}%
      \def\colorgray#1{\color[gray]{#1}}%
      \expandafter\def\csname LTw\endcsname{\color{white}}%
      \expandafter\def\csname LTb\endcsname{\color{black}}%
      \expandafter\def\csname LTa\endcsname{\color{black}}%
      \expandafter\def\csname LT0\endcsname{\color[rgb]{1,0,0}}%
      \expandafter\def\csname LT1\endcsname{\color[rgb]{0,1,0}}%
      \expandafter\def\csname LT2\endcsname{\color[rgb]{0,0,1}}%
      \expandafter\def\csname LT3\endcsname{\color[rgb]{1,0,1}}%
      \expandafter\def\csname LT4\endcsname{\color[rgb]{0,1,1}}%
      \expandafter\def\csname LT5\endcsname{\color[rgb]{1,1,0}}%
      \expandafter\def\csname LT6\endcsname{\color[rgb]{0,0,0}}%
      \expandafter\def\csname LT7\endcsname{\color[rgb]{1,0.3,0}}%
      \expandafter\def\csname LT8\endcsname{\color[rgb]{0.5,0.5,0.5}}%
    \else
      \def\colorrgb#1{\color{black}}%
      \def\colorgray#1{\color[gray]{#1}}%
      \expandafter\def\csname LTw\endcsname{\color{white}}%
      \expandafter\def\csname LTb\endcsname{\color{black}}%
      \expandafter\def\csname LTa\endcsname{\color{black}}%
      \expandafter\def\csname LT0\endcsname{\color{black}}%
      \expandafter\def\csname LT1\endcsname{\color{black}}%
      \expandafter\def\csname LT2\endcsname{\color{black}}%
      \expandafter\def\csname LT3\endcsname{\color{black}}%
      \expandafter\def\csname LT4\endcsname{\color{black}}%
      \expandafter\def\csname LT5\endcsname{\color{black}}%
      \expandafter\def\csname LT6\endcsname{\color{black}}%
      \expandafter\def\csname LT7\endcsname{\color{black}}%
      \expandafter\def\csname LT8\endcsname{\color{black}}%
    \fi
  \fi
    \setlength{\unitlength}{0.0500bp}%
    \ifx\gptboxheight\undefined%
      \newlength{\gptboxheight}%
      \newlength{\gptboxwidth}%
      \newsavebox{\gptboxtext}%
    \fi%
    \setlength{\fboxrule}{0.5pt}%
    \setlength{\fboxsep}{1pt}%
    \definecolor{tbcol}{rgb}{1,1,1}%
\begin{picture}(4320.00,3960.00)%
    \gplgaddtomacro\gplbacktext{%
      \csname LTb\endcsname
      \put(219,538){\makebox(0,0)[r]{\strut{}$10^{-14}$}}%
      \csname LTb\endcsname
      \put(219,1392){\makebox(0,0)[r]{\strut{}$10^{-13}$}}%
      \csname LTb\endcsname
      \put(219,2246){\makebox(0,0)[r]{\strut{}$10^{-12}$}}%
      \csname LTb\endcsname
      \put(219,3100){\makebox(0,0)[r]{\strut{}$10^{-11}$}}%
      \csname LTb\endcsname
      \put(338,174){\makebox(0,0){\strut{}$32$}}%
      \csname LTb\endcsname
      \put(1327,174){\makebox(0,0){\strut{}$64$}}%
      \csname LTb\endcsname
      \put(2317,174){\makebox(0,0){\strut{}$128$}}%
      \csname LTb\endcsname
      \put(3307,174){\makebox(0,0){\strut{}$256$}}%
      \csname LTb\endcsname
      \put(4297,174){\makebox(0,0){\strut{}$512$}}%
    }%
    \gplgaddtomacro\gplfronttext{%
      \csname LTb\endcsname
      \put(728,3608){\makebox(0,0)[l]{\strut{}Sparse Winograd, by \cite{Karstadt:2017aa}}}%
      \csname LTb\endcsname
      \put(728,3434){\makebox(0,0)[l]{\strut{}Winograd \cite{Winograd:1977:complexite}}}%
      \csname LTb\endcsname
      \put(728,3260){\makebox(0,0)[l]{\strut{}Sparse Strassen}}%
      \csname LTb\endcsname
      \put(728,3085){\makebox(0,0)[l]{\strut{}Strassen \cite{strassen:1969}}}%
      \csname LTb\endcsname
      \put(728,2911){\makebox(0,0)[l]{\strut{}Sparse \cref{alg:asopt} (\cref{alg:CoBschwartzopt})}}%
      \csname LTb\endcsname
      \put(728,2737){\makebox(0,0)[l]{\strut{}\cref{alg:asopt,eq:asopt}}}%
      \csname LTb\endcsname
      \put(728,2562){\makebox(0,0)[l]{\strut{}Conventional}}%
      \csname LTb\endcsname
      \put(0,3887){\makebox(0,0){\strut{}error}}%
      \csname LTb\endcsname
      \put(2296,0){\makebox(0,0){\strut{}square matrix dimension}}%
    }%
    \gplbacktext
    \put(0,0){\includegraphics[width={216.00bp},height={198.00bp}]{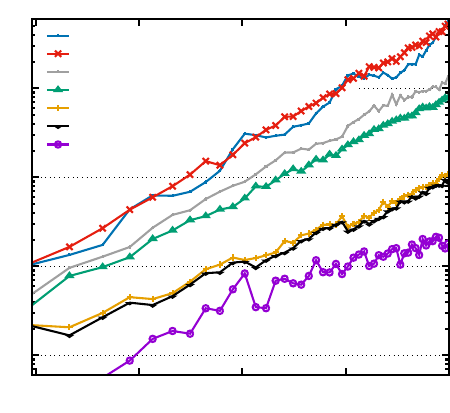}}%
    \gplfronttext
  \end{picture}%
\endgroup
\end{figure}
\par
Following~\cite[\S~3.2]{Ozaki:2012:erfMM}, we can also confirm our algorithms' accuracy on badly conditioned matrices (see~\cref{fig:badcond} in~\Cref{seq:furtherNumericalExperiments}).
\begin{remark}
To further improve their practical behavior, as done in~\cite[\S~4.3]{demmel:2007b},~\cite[\S~6.1]{ballard:2012a} or~\cite[\S~6]{BBDLS16}, some diagonal scaling adapted to specific input matrices can be added to any algorithms and thus to any of the variants presented here.
The idea is to precondition input matrices with well suited matrices thanks to \Cref{lem:sandwiching}.
\end{remark}
\section{Conclusion and future work}
We have presented a technique and a software to find more accurate recursive fast matrix multiplication algorithms.
Our analysis shows that our most accurate~\(\matrixsize{2}{2}\) formulas are probably optimal with respect to the tensor nuclear (Frobenius) norm.
We still anyway have a potential gap of at most 2.6\% to further explore.
\par
We also have shown ways to optimize the time complexity of~\(\matrixsize{2}{2}\) matrix product variants to simultaneously obtain a better accuracy in practice and a time complexity bound with the best currently known leading term (obtained via alternative basis sparsification).
There remains to compare the actual timings of these variants on different type of matrices and different range of matrix dimensions.
\par
Also, isotropies play a central role in this matter as shown by the fact that the minimal growth factor reached in this work is exactly the same as that of the algorithm obtained by~\cite[Eq.~(22)]{Grochow:2016aa}, while reconstructing Strassen's algorithm, using only the knowledge of its stabilizer and its representation with minimal Frobenius norms.
\clearpage
\appendix%
\section{Supplementary materials}
\subsection{Computational proofs}\label{sec:ComputationalProofs}
We gather here proofs of several propositions that are
simple computations on objects presented in our work (verified also
in~\cite{jgd:2024:mFMM}).
\begin{proof}[Proof of~\cref{prop:BestGrowthFactor}]
To simplify our computations, we use the following coordinates~\({{(\rho,\xi)}={\bigl(\sqrt[4]{\lfrac{4}{r}},\lfrac{(x-1)}{2}\bigr)}}\), the matrix~\({u(\rho,\xi)}\) and the associated isotropy~\(\IsotropyAction{\Isotropy{g}_{\rho,\xi}}{\tensor{S}}\).
In that case, the explicit expression of the~\(\gamma_{2,1}\) growth factor~\({\gamma\bigl({\IsotropyAction{\Isotropy{g}_{r,x}}{\tensor{S}}}\bigr)}\)
along this orbit is~\({2\,\sqrt{2}+3\,{\mathcal{A}}}\) with~\({\mathcal{A}(r,x)}\) equal to~\(\lfrac{\bigl({(1+x)}^{2}\!+r\bigr)\bigl({(x-1)}^{2}\!+r\bigr)}{r\sqrt{r}}\).
To conclude, we prove that the minimum of~\({\mathcal{A}(r,x)}\) in~\({{\RR^{+}}\times{\RR}}\) is~\({\lfrac{16}{3\sqrt{3}}}\).
The partial derivatives of~\(\mathcal{A}(r,x)\) w.r.t.~\(r\) and~\(x\) are
\begin{equation}\label{eq:growthfactorgradient}
\frac{\partial\mathcal{A}}{\partial{x}}\!=\!4{\frac{\left({x}^{2}+r-1\right)x}{{r}^{3/2}}},\
\frac{\partial\mathcal{A}}{\partial{r}}\!=\!{\frac{{r}^{2}-2\left({x}^{2}+1\right)r-3{\left({x}^{2}-1\right)}^{2}}{2{r}^{5/2}}}.
\end{equation}
First, notice that~\({\frac{\partial\mathcal{A}}{\partial{x}}\!\left(1-x^2,x\right)}\) is~\(0\) and that~\(\frac{\partial\mathcal{A}}{\partial{r}}\!\left(1-x^2,x\right)\) is equal to~\({\lfrac{2}{{\bigl((x-1)(1+x)\bigr)}{}^{3/2}}}\).
The only critical point is~\({x}\) equal to~\(0\) and, as~\(r\) is positive, it only could be equal to~\(3\).
The Hessian matrix is:
\begin{equation}
H(\mathcal{A}(r,x))=
\frac{1}{r^{\frac{3}{2}}}\begin{smatrix}
4\,\left({3\,{x}^{2}+r-1}\right)&%
{-\frac{2}{r}}\,\left({3\,{x}^{2}+r-3}\right)\\%
-\frac{2}{r}\,\left({3\,{x}^{2}+r-3}\right)&%
-{\frac{r^2-6\,\left(x^{2}+1\right)r-15\,{\left(x^{2}-1\right)}^{2}}{4r^{2}}}
\end{smatrix},
\end{equation}
one can notice that~\({H(\mathcal{A}(3,0))}\) is equal to~\({\lfrac{\begin{smatrix} 2^{3}&0\\ 0&\lfrac{2}{3} \end{smatrix}}{3\sqrt{3}}}\).
\par
Hence, the second partial derivative test states that~\({\bigl(\sqrt[4]{4/3},-1/2\bigr)}\) is a local minimum of~\(\GrowthFactor{\IsotropyAction{\Isotropy{g}_{r,x}}{\tensor{S}}}\); it is equal to~\({2\,\sqrt{2}+\lfrac{16}{\sqrt{3}}}\) that is~\({\mathord{\approx}{12.06603143}}\).
To conclude, the~\(\gamma_{2,1}\) growth factor reaches its global minimal at this point on the considered orbit because it is its only critical point.
\end{proof}
\begin{proof}[Proof of~\cref{prop:tencubetwonorm}]
Recall that any of the three matrices in the \textsc{hm} representation of such an algorithm is obtained by row and column permutations of one of these matrices, multiplied by the Kronecker product of two invertible~\(\matrixsize{2}{2}$ matrices~\(\mat{W}\) and~\(\mat{V}\) (see \Cref{lem:actionOnHMRepresentation} and \Cref{thm:IsotropiesActTransitivelyOnOptimalAlgorithm}).
From the analysis of~\cref{sec:numericalStabilityMeasure}, we only need to consider the case where~\(\mat{W}\) and~\(\mat{V}\) are each of the form~\({\begin{smatrix} r & x\\ 0 & r^{-1} \end{smatrix}}\), with strictly positive~\(r\) (matrices~\(\mat{W}\) and~\(\mat{V}\) are taken in a simpler form for the sake of simplicity).
For this, we let~\(\mat{W}\) be~\({\begin{smatrix} r & x\\ 0 & r^{-1} \end{smatrix}}\),~\(\mat{V}\) be~\({\begin{smatrix} s& y\\ 0 & s^{-1} \end{smatrix}}\) and choose~\(\mat{L}\) the first component of Strassen's \textsc{hm} representation given in~\cref{eq:StrassenHMRepresentation}.
The Frobenius norm of~\({\mat{L}\cdot(\mat{W}\tensorproduct{\mat{V}})}$ is given by the following function of~\({r,x,s}\) and~\(y\):
\begin{multline}
f(r,x,s,y)=4r^2s^2 + 3r^2y^2 + 3s^2x^2 + x^2y^2\\
+ \frac{r^2}{s^2} + \frac{s^2}{r^2}+\frac{1}{{(rs)}^2}
+ {\left(\frac{x}{s} + \frac{1}{rs}\right)}^{\!2}
+ {\left(\frac{y}{r} - \frac{1}{rs}\right)}^{\!2}
+ {\left(\frac{x}{s} - xy\right)}^{\!2}
\\
+ {\left(\frac{y}{r} + xy\right)}^{\!2}
+{\left(xy + \frac{1}{rs}\right)}^{\!2}
+{\left(\frac{s}{r} + xs\right)}^{\!2}
+{\left(\frac{r}{s} - ry\right)}^{\!2}.
\end{multline}
Solving the four partial derivatives of the gradient~\({\left[\frac{\partial{f}}{\partial{r}}, \frac{\partial{f}}{\partial{s}}, \frac{\partial{f}}{\partial{x}}, \frac{\partial{f}}{\partial{y}}\right]}\), for a simultaneous zero, we obtain that the only real extrema of~\(f\) are at the four points~\({{r={\pm\sqrt[4]{\lfrac{3}{4}}}},{s={\pm{r}}},{x={\lfrac{2}{3}r^{3}}},{{y}={-\lfrac{2}{3}s^{3}}}}\), for which its value is always~$10$.
The additional constraint that $r$ and $s$ are positive, thus gives a single extremum.
Now, the Hessian matrix at that point is computed as:
\begin{equation}
{\mathbf{H}}_f\left(\frac{\sqrt[4]{3}}{\sqrt{2}},\frac{\sqrt[4]{3}}{\sqrt{6}},\frac{\sqrt[4]{3}}{\sqrt{2}},{-}\frac{\sqrt[4]{3}}{\sqrt{6}}\right)=\frac{4}{9}\begin{smatrix}
\frac{195}{\sqrt{3}}&\frac{69}{\sqrt{3}}&-15&15\\
\frac{69}{\sqrt{3}}&\frac{195}{\sqrt{3}}&-15&15\\
-15&-15&\frac{45}{\sqrt{3}}&\frac{9}{\sqrt{3}}\\
15&15&\frac{9}{\sqrt{3}}&\frac{45}{\sqrt{3}}
\end{smatrix}.
\end{equation}
Leaving out the factor~\(\lfrac{4}{9}\), the characteristic polynomial
of this matrix
is~\({Z^4-160\sqrt{3}Z^3+22536Z^2-362880\sqrt{3}Z+5143824}\) whose all
four roots~\({{\sqrt{7500}\pm\sqrt{5232}},{30\pm{12}\sqrt{3}}}\) are
positive.
Therefore the point is a local minimum.
From this, we see that~\(\sqrt{10}{}^3\) is a lower bound on the product of their Frobenius norms.
Furthermore, remark that this lower bound is reached, by the algorithm
which \textsc{hm} representation is given in~\Cref{eq:asopt}.
\end{proof}
\subsection{Rational approximations}\label{sec:rationalapproximations}
As stated in \Cref{rem:rationalHigherOrderApproximations}, by approximating the minimal point of the~\(\gamma_{2,1}\) growth factor presented in \Cref{prop:BestGrowthFactor}, we could construct further algorithms presented in this section.
\begin{equation}\label{eq:powrot}
\begin{smatrix}
\frac{4}{9}&{-}\frac{8}{9}&{-}\frac{8}{9}&{-}\frac{4}{9}\\
0&\frac{5}{9}&0&\frac{10}{9}\\
\frac{8}{9}&{-}\frac{2}{3}&0&0\\
\frac{4}{9}&\frac{2}{9}&\frac{8}{9}&\frac{4}{9}\\
0&{-}\frac{10}{9}&0&0\\
\frac{4}{9}&{-}\frac{1}{3}&{-}\frac{8}{9}&\frac{2}{3}\\
{-}\frac{4}{9}&{-}\frac{2}{9}&\frac{8}{9}&\frac{4}{9}\\
\end{smatrix};\
\begin{smatrix}
{-}\frac{3}{5}&\frac{4}{5}&{-}\frac{4}{5}&{-}\frac{3}{5}\\
0&\frac{1}{2}&0&-1\\
-1&\frac{1}{2}&0&0\\
0&\frac{5}{4}&0&0\\
\frac{3}{5}&{-}\frac{3}{10}&\frac{4}{5}&{-}\frac{2}{5}\\
\frac{2}{5}&\frac{3}{10}&{-}\frac{4}{5}&{-}\frac{3}{5}\\
{-}\frac{3}{5}&{-}\frac{9}{20}&{-}\frac{4}{5}&{-}\frac{3}{5}\\
\end{smatrix};\
\Transpose{\begin{smatrix}
\frac{9}{20}&\frac{9}{10}&\frac{9}{10}&{-}\frac{9}{20}\\
0&0&\frac{27}{40}&{-}\frac{9}{10}\\
{-}\frac{9}{8}&0&{-}\frac{9}{16}&0\\
\frac{9}{20}&\frac{9}{10}&\frac{9}{40}&\frac{9}{20}\\
{-}\frac{27}{40}&\frac{9}{10}&\frac{27}{80}&{-}\frac{9}{20}\\
0&0&{-}\frac{9}{8}&0\\
\frac{9}{20}&\frac{9}{10}&{-}\frac{9}{40}&{-}\frac{9}{20}
\end{smatrix}}.
\end{equation}
First, \cref{eq:powrot} is an orthogonal optimization of~\cref{eq:powers}, with one canonical vector in each of components of the \textsc{hm} representation.
Unfortunately some small non-powers of~\(2\) are then unavoidable, but this gives in~\cref{alg:powrot} an algorithm realizing the formula with fewer additions than that of~\cref{alg:powers}.
\begin{table}[h]%
\centering\fbox{\begin{minipage}{.95\linewidth}\vspace{-5pt}
\[\setlength\arraycolsep{4pt}\renewcommand{\arraystretch}{1}
\begin{array}{llll}
u_1{=}\frac{1}{2}a_{12}{+}a_{22}&
t_1{=}\frac{10}{9}u_1&
t_2{=}\frac{8}{9}a_{11}{-}\frac{2}{3}a_{12}&
t_4{=}\frac{10}{9}a_{12}\\[\smallskipamount]
t_3{=}\frac{8}{9}a_{21}{+}\frac{4}{9}\bigl(a_{11}{+}u_1\bigr)&
t_0{=}t_2{-}t_3&
t_5{=}t_1{+}t_0&
t_6{=}t_4{+}t_0\\[\smallskipamount]
v_1{=}\frac{1}{2}b_{12}&
s_1{=}v_1{-}b_{22}&
s_2{=}v_1{-}b_{11}&
s_3{=}\frac{5}{4}b_{12}\\[\smallskipamount]
s_4{=}\frac{2}{5}b_{22}{-}\frac{4}{5}b_{21}{+}\frac{3}{5}s_2&
s_0{=}s_1{+}s_4&
s_5{=}s_0{-}s_2&
s_6{=}s_3{-}s_0
\end{array}
\]
\vspace{-4pt}
\[%
\begin{array}{llll}
p_0{=}t_0{\cdot}s_0&
p_1{=}t_1{\cdot}s_1&
p_2{=}t_2{\cdot}s_2&
p_3{=}t_3{\cdot}s_3\\
p_4{=}t_4{\cdot}s_4&
p_5{=}t_5{\cdot}s_5&
p_6{=}t_6{\cdot}s_6
\end{array}
\]
\vspace{-4pt}
\[%
\begin{array}{llllll}
w_1{=}p_6{+}p_0{+}p_4&
w_2{=}p_5{+}p_6&
w_3{=}p_3{+}w_1&
w_4{=}p_2{+}p_4\\
w_5{=}p_1{+}w_1&
w_6{=}\frac{9}{20}w_3&&
c_{11}{=}w_6{-}\frac{9}{8}w_4\\[\smallskipamount]
c_{12}{=}\frac{9}{10}w_3&
\multicolumn{2}{l}{c_{21}{=}\frac{27}{40}w_5{-}\frac{9}{8}w_2{+}\frac{1}{2}c_{11}}&
c_{22}{=}w_6{-}\frac{9}{10}w_5
\end{array}
\]
\vspace{-5pt}
\end{minipage}}
\caption{\textsc{slp} of~\cref{eq:powrot},~\({\gamma_{2,1}\approx{12.2034}}\), with~\(24\) add.\ and~\(19\) mul.}\label{alg:powrot}%
\end{table}
\par
Finally, we present in~\cref{alg:0695,alg:0661}, successive higher-order rational approximations of the point~\({\bigl(\sqrt[4]{\lfrac{4}{3}},-\lfrac{1}{2}\bigr)}\) reducing the growth factor~\(\gamma_{2,1}\) to~\(12.0695\) (resp.~\(12.0661\)), approaching~\(12.06603\).
They then provide rational algorithms whose accuracy is pretty close to our best one, as shown in~\cref{fig:gamma}.
\begin{equation}\label{alg:0695}
\begin{aligned}
\begin{smatrix}
{-}\frac{167042}{345665} &  \frac{295936}{345665} &  {-}\frac{295936}{345665} &  {-}\frac{167042}{345665}\\
{-}\frac{178623}{345665} &  {-}\frac{51622047}{176980480} &  \frac{295936}{345665} &  \frac{167042}{345665}\\
0 &  {-}\frac{51622047}{88490240} &  0 &  \frac{334084}{345665}\\
-1 &  \frac{289}{512} &  0 &  0\\
0 &  \frac{289}{256} &  0 &  0\\
{-}\frac{167042}{345665} &  {-}\frac{24137569}{88490240} &  {-}\frac{295936}{345665} &  {-}\frac{167042}{345665}\\
{-}\frac{167042}{345665} &  \frac{24137569}{88490240} &  {-}\frac{295936}{345665} &  \frac{167042}{345665}
\end{smatrix};\
\begin{smatrix}
{-}\frac{256}{289} &  {-}\frac{1}{2} &  \frac{1}{2} &  {-}\frac{256}{289}\\
{-}\frac{345665}{295936} &  0 &  0 &  0\\
{-}\frac{345665}{591872} &  0 &  \frac{345665}{334084} &  0\\
{-}\frac{178623}{295936} &  -1 &  0 &  0\\
\frac{178623}{591872} &  \frac{1}{2} &  \frac{178623}{334084} &  \frac{256}{289}\\
{-}\frac{289}{1024} &  \frac{1}{2} &  {-}\frac{1}{2} &  \frac{256}{289}\\
{-}\frac{289}{1024} &  \frac{1}{2} &  \frac{1}{2} &  {-}\frac{256}{289}
\end{smatrix};\\
\begin{smatrix}
\frac{295936}{345665}&\frac{295936}{345665}&0&0&\frac{295936}{345665}&\frac{295936}{345665}&0\\
\frac{178623}{345665}&{-}\frac{167042}{345665}&1&0&\frac{178623}{345665}&\frac{178623}{345665}&0\\
{-}\frac{178623}{345665}&{-}\frac{178623}{345665}&0&1&\frac{167042}{345665}&{-}\frac{178623}{345665}&0\\
\frac{295936}{345665}&\frac{51622047}{176980480}&\frac{289}{512}&{-}\frac{289}{512}&\frac{51622047}{176980480}&{-}\frac{31906176129}{102294717440}&{-}\frac{345665}{295936}\\
\end{smatrix}.
\end{aligned}
\end{equation}
\begin{equation}\label{alg:0661}
\begin{aligned}
\begin{smatrix}
\frac{33124}{38165}&\frac{19208}{38165}&{-}\frac{19208}{38165}&\frac{33124}{38165}\\
\frac{33124}{38165}&\frac{19208}{38165}&\frac{18957}{38165}&\frac{1857786}{6449885}\\
0&\frac{38416}{38165}&0&\frac{3715572}{6449885}\\
0&0&1&{-}\frac{98}{169}\\
0&0&0&\frac{196}{169}\\
\frac{33124}{38165}&\frac{19208}{38165}&{-}\frac{19208}{38165}&{-}\frac{1882384}{6449885}\\
\frac{33124}{38165}&{-}\frac{19208}{38165}&{-}\frac{19208}{38165}&\frac{1882384}{6449885}\\
\end{smatrix};\quad
\begin{smatrix}
{-}\frac{169}{196}&{-}\frac{1}{2}&\frac{1}{2}&{-}\frac{169}{196}\\
\frac{38165}{33124}&0&0&0\\
\frac{38165}{66248}&0&{-}\frac{38165}{38416}&0\\
\frac{18957}{33124}&1&0&0\\
\frac{18957}{66248}&\frac{1}{2}&\frac{18957}{38416}&\frac{169}{196}\\
{-}\frac{49}{169}&\frac{1}{2}&{-}\frac{1}{2}&\frac{169}{196}\\
{-}\frac{49}{169}&\frac{1}{2}&\frac{1}{2}&{-}\frac{169}{196}\\
\end{smatrix};\\
\begin{smatrix}
{-}\frac{18957}{38165} & \frac{19208}{38165} & -1 & 0 & {-}\frac{18957}{38165} & {-}\frac{18957}{38165} & 0 \\
{-}\frac{33124}{38165} & {-}\frac{1857786}{6449885} & {-}\frac{98}{169} & \frac{98}{169} & {-}\frac{1857786}{6449885} & \frac{359367849}{1264177460} & \frac{38165}{33124} \\
\frac{33124}{38165} & \frac{33124}{38165} & 0 & 0 & \frac{33124}{38165} & \frac{33124}{38165} & 0 \\
{-}\frac{18957}{38165} & {-}\frac{18957}{38165} & 0 & 1 & \frac{19208}{38165} & {-}\frac{18957}{38165} & 0 \\
\end{smatrix}.
\end{aligned}
\end{equation}
\subsection{Further numerical experiments}\label{seq:furtherNumericalExperiments}
Following~\cite[\S~3.2]{Ozaki:2012:erfMM}, we study in \cref{fig:badcond} the effect of sparsification on random matrix with preassigned singular values and large condition number~\({\mathord{\approx}{10^{12}}}\) given by the Matlab function gallery 'randsvd'.
The fast variants behavior is unchanged while only the conventional algorithm performs better.
\begin{figure}[htbp]
\caption{Numerical Effect of Sparsification (\footnotesize{large conditioning})}\label{fig:badcond}
\Description[Numerical Effect of Sparsification]{Numerical Effect of Sparsification}
%
\begingroup
  \makeatletter
  \providecommand\color[2][]{%
    \GenericError{(gnuplot) \space\space\space\@spaces}{%
      Package color not loaded in conjunction with
      terminal option `colourtext'%
    }{See the gnuplot documentation for explanation.%
    }{Either use 'blacktext' in gnuplot or load the package
      color.sty in LaTeX.}%
    \renewcommand\color[2][]{}%
  }%
  \providecommand\includegraphics[2][]{%
    \GenericError{(gnuplot) \space\space\space\@spaces}{%
      Package graphicx or graphics not loaded%
    }{See the gnuplot documentation for explanation.%
    }{The gnuplot epslatex terminal needs graphicx.sty or graphics.sty.}%
    \renewcommand\includegraphics[2][]{}%
  }%
  \providecommand\rotatebox[2]{#2}%
  \@ifundefined{ifGPcolor}{%
    \newif\ifGPcolor
    \GPcolortrue
  }{}%
  \@ifundefined{ifGPblacktext}{%
    \newif\ifGPblacktext
    \GPblacktexttrue
  }{}%
  \let\gplgaddtomacro\g@addto@macro
  \gdef\gplbacktext{}%
  \gdef\gplfronttext{}%
  \makeatother
  \ifGPblacktext
    \def\colorrgb#1{}%
    \def\colorgray#1{}%
  \else
    \ifGPcolor
      \def\colorrgb#1{\color[rgb]{#1}}%
      \def\colorgray#1{\color[gray]{#1}}%
      \expandafter\def\csname LTw\endcsname{\color{white}}%
      \expandafter\def\csname LTb\endcsname{\color{black}}%
      \expandafter\def\csname LTa\endcsname{\color{black}}%
      \expandafter\def\csname LT0\endcsname{\color[rgb]{1,0,0}}%
      \expandafter\def\csname LT1\endcsname{\color[rgb]{0,1,0}}%
      \expandafter\def\csname LT2\endcsname{\color[rgb]{0,0,1}}%
      \expandafter\def\csname LT3\endcsname{\color[rgb]{1,0,1}}%
      \expandafter\def\csname LT4\endcsname{\color[rgb]{0,1,1}}%
      \expandafter\def\csname LT5\endcsname{\color[rgb]{1,1,0}}%
      \expandafter\def\csname LT6\endcsname{\color[rgb]{0,0,0}}%
      \expandafter\def\csname LT7\endcsname{\color[rgb]{1,0.3,0}}%
      \expandafter\def\csname LT8\endcsname{\color[rgb]{0.5,0.5,0.5}}%
    \else
      \def\colorrgb#1{\color{black}}%
      \def\colorgray#1{\color[gray]{#1}}%
      \expandafter\def\csname LTw\endcsname{\color{white}}%
      \expandafter\def\csname LTb\endcsname{\color{black}}%
      \expandafter\def\csname LTa\endcsname{\color{black}}%
      \expandafter\def\csname LT0\endcsname{\color{black}}%
      \expandafter\def\csname LT1\endcsname{\color{black}}%
      \expandafter\def\csname LT2\endcsname{\color{black}}%
      \expandafter\def\csname LT3\endcsname{\color{black}}%
      \expandafter\def\csname LT4\endcsname{\color{black}}%
      \expandafter\def\csname LT5\endcsname{\color{black}}%
      \expandafter\def\csname LT6\endcsname{\color{black}}%
      \expandafter\def\csname LT7\endcsname{\color{black}}%
      \expandafter\def\csname LT8\endcsname{\color{black}}%
    \fi
  \fi
    \setlength{\unitlength}{0.0500bp}%
    \ifx\gptboxheight\undefined%
      \newlength{\gptboxheight}%
      \newlength{\gptboxwidth}%
      \newsavebox{\gptboxtext}%
    \fi%
    \setlength{\fboxrule}{0.5pt}%
    \setlength{\fboxsep}{1pt}%
    \definecolor{tbcol}{rgb}{1,1,1}%
\begin{picture}(4320.00,4160.00)%
    \gplgaddtomacro\gplbacktext{%
      \csname LTb\endcsname
      \put(219,990){\makebox(0,0)[r]{\strut{}$10^{-15}$}}%
      \csname LTb\endcsname
      \put(219,1743){\makebox(0,0)[r]{\strut{}$10^{-14}$}}%
      \csname LTb\endcsname
      \put(219,2495){\makebox(0,0)[r]{\strut{}$10^{-13}$}}%
      \csname LTb\endcsname
      \put(219,3247){\makebox(0,0)[r]{\strut{}$10^{-12}$}}%
      \csname LTb\endcsname
      \put(338,174){\makebox(0,0){\strut{}$32$}}%
      \csname LTb\endcsname
      \put(1327,174){\makebox(0,0){\strut{}$64$}}%
      \csname LTb\endcsname
      \put(2317,174){\makebox(0,0){\strut{}$128$}}%
      \csname LTb\endcsname
      \put(3307,174){\makebox(0,0){\strut{}$256$}}%
      \csname LTb\endcsname
      \put(4297,174){\makebox(0,0){\strut{}$512$}}%
    }%
    \gplgaddtomacro\gplfronttext{%
      \csname LTb\endcsname
      \put(728,3808){\makebox(0,0)[l]{\strut{}Sparse Winograd, by \cite{Karstadt:2017aa}}}%
      \csname LTb\endcsname
      \put(728,3634){\makebox(0,0)[l]{\strut{}Winograd \cite{Winograd:1977:complexite}}}%
      \csname LTb\endcsname
      \put(728,3460){\makebox(0,0)[l]{\strut{}Sparse Strassen}}%
      \csname LTb\endcsname
      \put(728,3285){\makebox(0,0)[l]{\strut{}Strassen \cite{strassen:1969}}}%
      \csname LTb\endcsname
      \put(728,3111){\makebox(0,0)[l]{\strut{}Sparse \cref{alg:asopt} (\cref{alg:CoBschwartzopt})}}%
      \csname LTb\endcsname
      \put(728,2937){\makebox(0,0)[l]{\strut{}\cref{alg:asopt,eq:asopt}}}%
      \csname LTb\endcsname
      \put(728,2762){\makebox(0,0)[l]{\strut{}Conv.}}%
      \csname LTb\endcsname
      \put(0,4074){\makebox(0,0){\strut{}error}}%
      \csname LTb\endcsname
      \put(2296,0){\makebox(0,0){\strut{}square matrix dimension}}%
    }%
    \gplbacktext
    \put(0,0){\includegraphics[width={216.00bp},height={208.00bp}]{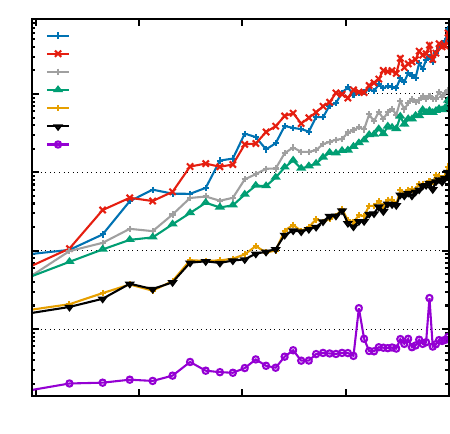}}%
    \gplfronttext
  \end{picture}%
\endgroup
\end{figure}

We here show more evidence on the practical accuracy of the algorithms, with respect to their~\(\gamma_{2,1}\) growth factor.
\Cref{fig:gammaunif} compares the main possibilities on a
uniform~\([-1,1]\) distribution, while~\cref{fig:gamma} was using a
normal distribution. The behavior is similar, with again our best
variant one or two orders of magnitude more accurate, and being quite
close to that of the conventional algorithm.
\begin{figure}[htbp]
\caption{Numerical accuracy for uniform [-1,1]
  distribution}\label{fig:gammaunif}
%
\begingroup
  \makeatletter
  \providecommand\color[2][]{%
    \GenericError{(gnuplot) \space\space\space\@spaces}{%
      Package color not loaded in conjunction with
      terminal option `colourtext'%
    }{See the gnuplot documentation for explanation.%
    }{Either use 'blacktext' in gnuplot or load the package
      color.sty in LaTeX.}%
    \renewcommand\color[2][]{}%
  }%
  \providecommand\includegraphics[2][]{%
    \GenericError{(gnuplot) \space\space\space\@spaces}{%
      Package graphicx or graphics not loaded%
    }{See the gnuplot documentation for explanation.%
    }{The gnuplot epslatex terminal needs graphicx.sty or graphics.sty.}%
    \renewcommand\includegraphics[2][]{}%
  }%
  \providecommand\rotatebox[2]{#2}%
  \@ifundefined{ifGPcolor}{%
    \newif\ifGPcolor
    \GPcolortrue
  }{}%
  \@ifundefined{ifGPblacktext}{%
    \newif\ifGPblacktext
    \GPblacktexttrue
  }{}%
  \let\gplgaddtomacro\g@addto@macro
  \gdef\gplbacktext{}%
  \gdef\gplfronttext{}%
  \makeatother
  \ifGPblacktext
    \def\colorrgb#1{}%
    \def\colorgray#1{}%
  \else
    \ifGPcolor
      \def\colorrgb#1{\color[rgb]{#1}}%
      \def\colorgray#1{\color[gray]{#1}}%
      \expandafter\def\csname LTw\endcsname{\color{white}}%
      \expandafter\def\csname LTb\endcsname{\color{black}}%
      \expandafter\def\csname LTa\endcsname{\color{black}}%
      \expandafter\def\csname LT0\endcsname{\color[rgb]{1,0,0}}%
      \expandafter\def\csname LT1\endcsname{\color[rgb]{0,1,0}}%
      \expandafter\def\csname LT2\endcsname{\color[rgb]{0,0,1}}%
      \expandafter\def\csname LT3\endcsname{\color[rgb]{1,0,1}}%
      \expandafter\def\csname LT4\endcsname{\color[rgb]{0,1,1}}%
      \expandafter\def\csname LT5\endcsname{\color[rgb]{1,1,0}}%
      \expandafter\def\csname LT6\endcsname{\color[rgb]{0,0,0}}%
      \expandafter\def\csname LT7\endcsname{\color[rgb]{1,0.3,0}}%
      \expandafter\def\csname LT8\endcsname{\color[rgb]{0.5,0.5,0.5}}%
    \else
      \def\colorrgb#1{\color{black}}%
      \def\colorgray#1{\color[gray]{#1}}%
      \expandafter\def\csname LTw\endcsname{\color{white}}%
      \expandafter\def\csname LTb\endcsname{\color{black}}%
      \expandafter\def\csname LTa\endcsname{\color{black}}%
      \expandafter\def\csname LT0\endcsname{\color{black}}%
      \expandafter\def\csname LT1\endcsname{\color{black}}%
      \expandafter\def\csname LT2\endcsname{\color{black}}%
      \expandafter\def\csname LT3\endcsname{\color{black}}%
      \expandafter\def\csname LT4\endcsname{\color{black}}%
      \expandafter\def\csname LT5\endcsname{\color{black}}%
      \expandafter\def\csname LT6\endcsname{\color{black}}%
      \expandafter\def\csname LT7\endcsname{\color{black}}%
      \expandafter\def\csname LT8\endcsname{\color{black}}%
    \fi
  \fi
    \setlength{\unitlength}{0.0500bp}%
    \ifx\gptboxheight\undefined%
      \newlength{\gptboxheight}%
      \newlength{\gptboxwidth}%
      \newsavebox{\gptboxtext}%
    \fi%
    \setlength{\fboxrule}{0.5pt}%
    \setlength{\fboxsep}{1pt}%
    \definecolor{tbcol}{rgb}{1,1,1}%
\begin{picture}(4320.00,3600.00)%
    \gplgaddtomacro\gplbacktext{%
      \csname LTb\endcsname
      \put(219,962){\makebox(0,0)[r]{\strut{}$10^{-15}$}}%
      \csname LTb\endcsname
      \put(219,1603){\makebox(0,0)[r]{\strut{}$10^{-14}$}}%
      \csname LTb\endcsname
      \put(219,2244){\makebox(0,0)[r]{\strut{}$10^{-13}$}}%
      \csname LTb\endcsname
      \put(219,2884){\makebox(0,0)[r]{\strut{}$10^{-12}$}}%
      \csname LTb\endcsname
      \put(338,174){\makebox(0,0){\strut{}$32$}}%
      \csname LTb\endcsname
      \put(1327,174){\makebox(0,0){\strut{}$64$}}%
      \csname LTb\endcsname
      \put(2317,174){\makebox(0,0){\strut{}$128$}}%
      \csname LTb\endcsname
      \put(3307,174){\makebox(0,0){\strut{}$256$}}%
      \csname LTb\endcsname
      \put(4297,174){\makebox(0,0){\strut{}$512$}}%
    }%
    \gplgaddtomacro\gplfronttext{%
      \csname LTb\endcsname
      \put(0,3445){\makebox(0,0){\strut{}error}}%
      \csname LTb\endcsname
      \put(2296,0){\makebox(0,0){\strut{}square matrix dimension}}%
      \csname LTb\endcsname
      \put(728,3248){\makebox(0,0)[l]{\strut{}Sparse Winograd, by \cite{Karstadt:2017aa}}}%
      \csname LTb\endcsname
      \put(728,3074){\makebox(0,0)[l]{\strut{}Winograd \cite{Winograd:1977:complexite}}}%
      \csname LTb\endcsname
      \put(728,2900){\makebox(0,0)[l]{\strut{}Sparse Strassen}}%
      \csname LTb\endcsname
      \put(728,2725){\makebox(0,0)[l]{\strut{}Strassen \cite{strassen:1969}}}%
      \csname LTb\endcsname
      \put(728,2551){\makebox(0,0)[l]{\strut{}\cref{alg:CoBschwartzopt}}}%
      \csname LTb\endcsname
      \put(728,2377){\makebox(0,0)[l]{\strut{}\cref{alg:asopt,eq:asopt}}}%
      \csname LTb\endcsname
      \put(728,2202){\makebox(0,0)[l]{\strut{}Conv.}}%
    }%
    \gplbacktext
    \put(0,0){\includegraphics[width={216.00bp},height={180.00bp}]{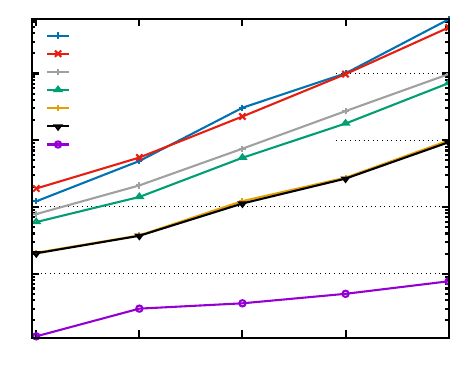}}%
    \gplfronttext
  \end{picture}%
\endgroup
\end{figure}

%
\subsection{Cancellation-free search}\label{app:suppl}
\Cref{alg:factoringout} describes a common sub-expression elimination heuristic that reduced the number of operations in our algorithms.
\begin{footnotesize}
\begin{algorithm}
\caption{Cancellation-free optimization of a linear operator}\label{alg:factoringout}
\begin{algorithmic}[1]
\Require{$\mat{M}\in\Field^{m{\times}n}$.}
\Ensure{A straight-line program computing $x\rightarrow{\mat{M}}{\cdot}x$.}
\Repeat\hfill\Comment{Precomputing all repeated pairs}
\State{In each row list all pairs of indices of non-zero coefficients;}
\State{Among all the rows, find the pair(s) with the maximal number of co-linear representatives;}
\State{In case of ties, exhaust all the possibilities with maximal pairs (or choose one using a score like that of~\cite[\S~3.2]{Boyar:2013aa});}
\State{Precompute the chosen pair (in a temporary variable);}
\State{Factor this pair out of all the rows: that is removing the pair from all rows but add a new column to the matrix (representing that pair) with the co-linear multiple of that temporary variable;}
\Until{no pair has more than $1$ representative}
\Statex\Comment{Multipliers by columns:}
\ForAll{equal coefficients in a column (up to sign)}
\State{Compute the product by the absolute value in a temporary variable;}
\State{Factor this coefficient out: remove it from the column, add a new column (representing that product) with a~\(\pm{1}\) in the corresponding row(s);}
\EndFor
\Statex\Comment{Multipliers by rows:}
\ForAll{equal coefficients in a row (up to sign)}
\State{Compute the sum (or subtraction) of variables with that same coefficients in a temporary variable;}
\State{Factor the coefficient out: remove it from the row, but add a new column (representing that sum/subtraction) with the coefficient in the same row;}
\EndFor
\Statex\Comment{Now the matrix has been simplified}
\State{Apply the remaining linear operations of the matrix.}
\end{algorithmic}
\end{algorithm}
\end{footnotesize}
\end{document}